\documentclass[12pt]{amsart}
\pdfoutput=1

\usepackage{newlfont}
\usepackage{subfigure}
\usepackage{pict2e}
\usepackage[percent]{overpic} 

\theoremstyle{plain}
\newtheorem{Thm}{Theorem}[section]
\newtheorem{Lem}[Thm]{Lemma}

\newtheorem*{LP}{Layout Process}

\newtheorem*{definition}{Definition}

\theoremstyle{definition}

\newcommand{\bbb}{\mathbb}      
\newcommand{\bC}{\bbb C}         
\newcommand{\bD}{\bbb D}         
\newcommand{\bI}{\bbb I}         
\newcommand{\bP}{\bbb P}         
\newcommand{\bQ}{\bbb Q}         
\newcommand{\bR}{\bbb R}         

\newcommand{\ccC}{\mathcal C}
\newcommand{\ccF}{\mathcal F}
\newcommand{\ccP}{\mathcal P}   
\newcommand{\ccV}{\mathcal V}
\newcommand{\ccU}{\mathcal U}

\newcommand{\cP}{P}             

\newcommand{\fp}{\mathfrak p}  
\newcommand{\fs}{\mathfrak s}   
\newcommand{\fC}{\mathfrak C}   
\newcommand{\fF}{\mathfrak F}   
\newcommand{\fS}{\mathfrak S}   
\newcommand{\fU}{\mathfrak U}   

\newcommand{\pssi}[3]{\vec{#1}_{#2,#3}}

\newcommand{\mF}{\mathcal F} 
\newcommand{\F}[1]{Figure~\ref{F:#1}}

\begin{document}

\title[Discrete Schwarzian]
{A Discrete Schwarzian
  Derivative\\ via Circle Packing}
\author{Kenneth Stephenson}
     \address{University of Tennessee, Knoxville} 
     \email{kstephe2@utk.edu}



\begin{abstract} There exists an extensive and fairly comprehensive
  discrete analytic function theory which is based on circle packing. This
  paper introduces a faithful discrete analogue of the classical
  Schwarzian derivative to this theory and develops its basic
  properties. The motivation comes from the current lack of circle
  packing algorithms in spherical geometry, and the discrete
  Schwarzian derivative may provide for new approaches. A companion
  localized notion called an intrinsic schwarzian is also investigated.
  The main concrete results of the paper are limited to circle packing
  flowers. A parameterization by intrinsic schwarzians is established,
  providing an essential packing criterion for flowers. The paper closes
  with the study of special classes of flowers that occur in the circle
  packing literature. As usual in circle packing, there are pleasant 
  surprises at nearly every turn, so those not interested in circle
  packing theory may still enjoy the new and elementary geometry seen in
  these flowers.
\end{abstract}

\maketitle Classical complex analysis, and conformal geometry in
general, have long benefited from a fundamental M\"obius invariant
known as the Schwarzian derivative. Recent decades have seen the
emergence of a comprehensive discrete analytic function theory and
associated discrete conformal geometry based on circle packing. This
discrete theory displays deep and intimate connections to conformal
geometry, so it is natural to ask if it, too, could benefit from such
an invariant. This paper establishes definitions for a
discrete Schwarzian derivative and verifies fundamental properties
that are largely faithful to the classical version. It also introduces
a local M\"obius invariant, an intrinsic schwarzian, and begins to
lay out how these invariants might provide important tools in
advancing the theory of circle packing.

M\"obius or projective invariance is exemplified by quantities which remain
unchanged after application of M\"obius transformations.  While the
Riemann sphere $\bP$ is the native habitat for M\"obius actions, it is
also far and away the most challenging for circle packing. Indeed,
with few exceptions, circle packings on $\bP$ have been merely
stereographic projections of packings developed in the euclidean or
hyperbolic setting. These spherical difficulties account for perhaps
the most glaring gap in discrete analytic function theory, namely, the
ability to create and manipulate discrete rational functions.

The circle packing community has exhausted most approaches to working
in spherical geometry, with precious little to show for it.  Perhaps
discrete Schwarzian derivatives can provide the fresh perspective
needed to move forward. The reader should not expect miracles,
however. Although we do establish robust definitions and basic
properties for a discrete Schwarzian derivative, taking our lead from
pioneering work by Gerald Orick, and although we take the opening
steps, there are no breakthrough theoretical tools here. On the other
hand, in the experimental world available via circle packing, the
discrete Schwarzian derivative and the associated intrinsic schwarzian
open wholly new vistas for investigation. Concrete results here deal
mostly with the fundamental unit within every circle packing, namely,
the circle packing ''flower''. As invariably happens in circle
packing, both beautiful visualizations and beautiful formulas pop up
around every corner. Whether or not the reader is involved in circle
packing theory, there is much to appreciate in the surprising and
pleasing elementary geometry we encounter. And we can always be alert
for that breakout tool.

\vspace{10pt} Here is a brief overview of the paper: We first provide
necessary (but brief) background on circle packing, on the associated
discrete analytic functions, on geometry and M\"obius transformations,
and on the central role experiments play in this topic. In Section~2
we review the classical Schwarzian derivative and define a discrete
version for mappings between circle packings. Moving beyond that
direct analog, we extract a local version, an {\sl intrinsic
  schwarzian} attached to individual packings. A principal goal --- a
distant goal --- is methods for recognizing, creating, and ultimately
manipulating (intrinsic) schwarzians for packings. These schwarzians
form edge labels analogous to the vertex (i.e., radius) labels which
dominate the theory, but which largely fail on the sphere. Section~3
illustrates the as-yet-unfulfilled potential for schwarzians as a
mechanism for laying out circle packings.  The struggle to work with
discrete meromorphic functions is our main motivation, but results
could also apply to circle packings on projective surfaces.

We switch in Section~\ref{S:flowers} to the paper's modest results
from our opening skirmishes with schwarzians; namely, describing the
schwarzians for flowers, the elemental circle packings.  An $n$-flower
consists of a central circle surrounded by a chain of $n$ tangent
``petal'' circles. A flower is {\sl un-branched} if the petals wrap
once around the center and {\sl branched} if they wrap 2 or more
times. It is {\sl univalent} if un-branched and the petals have
mutually disjoint interiors. Using a mechanical layout process and
computations detailed in the Appendix we work our way through the
early cases $n=3,4,5$, and 6 to general flowers. We reach
characterizations of un-branched (Theorem~\ref{T:unbranched}) and
univalent (Theorem~\ref{T:univalent}) flowers and criteria for
branching.

We conclude the paper with Section~\ref{S:sclass} by applying what we
have learned to several special classes of flowers. These cases will
contribute only marginally to the larger campaign, but they raise our
spirits with beautiful geometric, visual, and arithmetic
features. And although much remains to be done, in the author's view
the results for flowers alone are worth the effort.

\section{Background} \label{SS:background}

\vspace{10pt} \subsection{On Circle Packing} \label{SS:bcp} A {\sl
  circle packing} is a configuration of circles satisfying a
prescribed pattern of tangencies. Circle packings and their
connections to comformal geometry were introduced by William Thurston
in 1985, \cite{wT85}. Circle packings exist in great profusion in
euclidean, hyperbolic, and spherical geometry and more recently on
surfaces with affine and projective structures, \cite{SSW12},
\cite{KMT03}. The principal reference for this paper is \cite{kS05}.

The fundamental machinery is quite straightforward: The pattern of
tangencies for a circle packing $P$ is encoded in an abstract
(simplicial) complex $K$, a triangulation of a topological
surface. There is a circle $C_v\in P$ associated with each vertex $v$
of $K$ and each edge $\langle v,w\rangle$ of $K$ indicates a required
tangency between circles $C_v$ and $C_w$. Note that every ``circle''
is associated with a interior, forming a topological disc.  Two
circles are (externally) tangent if they intersect in a single point
and their interiors are mutually disjoint. Often the key data
associated with a packing is a radius label $R$, which contains a
radius $R(v)$ for the circle associate with vertex $v\in K$.

Some basic terminology will be useful in the sequel: The circles of a
packing $P$ occur in mutally tangent triples $\{C_v,C_w,C_u\}$.  The
geodesics connecting the three centers pass through the three tangency
points and form a geometric face. This is a geometric triangle
associated with the abstract {\sl face} $\{v,w,u\}$ of $K$.  The
surface formed by the geometric faces is called the {\sl carrier} of
$P$.  The packing $P$ is {\sl univalent} if its circles have mutually
disjoint interiors.

The packing $P$ can also be viewed as a collection of interconnected
flowers: a {\sl flower} consists of a {\sl central} circle $C_v$ and
the chain of successively tangent {\sl petal} circles, 
$\{C_{v_0},\cdots,C_{v_{n-1}}\}$, all tangent to $C_v$.  
A flower is {\sl closed} if
$C_{v_{n-1}}$ is tangent to $C_{v_0}$, in which case $v$ is an {\sl
  interior} vertex of $K$, whereas a flower is {\sl open} if and only
if $v$ is a {\sl boundary} vertex. (To avoid pathologies, we require
of $K$ that every boundary vertex has at least one interior neighbor.)
There are three classes of closed flowers: A {\sl univalent}
flower is one whose petals have mutually disjoint interiors. An {\sl
  un-branched} flower is one whose petals wrap once around the center, possibly
with overlaps between non-contiguous petals. Finally, a {\sl branched} flower is
one whose petals wrap more than once about the center and its {\sl degree} $d$
is number of times it wraps.

A circle packing $P$ is {\sl univalent} if its circles have mutually
disjoint interior. It is necessary (but not sufficient that the flowers
for interior vertices are univalent). If an interior circle $C_v$ has
a branched flower, then we say that $P$ has a branch point at $v$.

The surprising richness of the topic is seen in the foundational
existence and uniqueness result; namely, the Koebe-Andreev-Thurston
(KAT) Theorem, which states that for any triangulation $K$ of a
topological sphere, there exists an associated univalent circle
packing $\ccP_K$ of the Riemann sphere $\bP$, and that $\ccP_K$ is
unique up to M\"obius transformations (and inversions) of
$\bP$. Thurston also proposed in \cite{wT85} a clever algorithm for
actually computing such packings, allowing us today to treat circle
packing as a verb: to "circle pack" a complex $K$ is to create and
manipulate associated circle packings.

\vspace{10pt}
\subsection{On Discrete Analytic Functions} \label{SS:bdaf}
Intriguing as circle packings were in their own right, it was a
conjecture in Thurston's talk that really
fired up the topic. An example in \F{RMT} will set the stage.
Let $P$ be a univalent circle packing filling a simply connected
region $\Omega$ of the plane, as on the right in the figure, and let
$K$ be the underlying complex.  Thurston proved using KAT that there
exists a univalent circle packing for $K$ in the unit disc $\bD$ whose
boundary circles are all horocycles, as on the left in the
figure. With two circle packings for the same complex, one may define
$F:\ccP_K\longrightarrow P$ by identifying corresponding circles.
Essentially a mapping from the unit disc to $\Omega$, $F$ is roughly
analogous to the classical Riemann Mapping.

\begin{figure}[h]
\begin{center}
\begin{overpic}[width=.7\textwidth
  ]{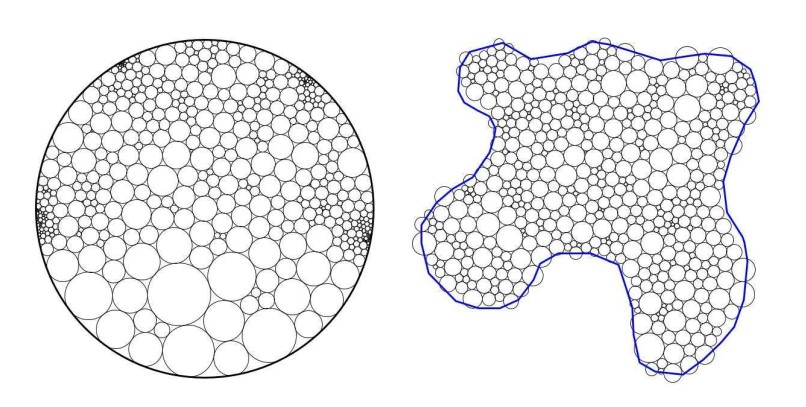}
  \put (36,45) {$\ccP_K$}
  \put (55,45) {$P$}
  \put(42,46.5){\vector(1.0,0){12}}
  \put (45,42) {$F$}
  \put (73,5) {$\Omega$}
  \put (5,5) {$\bD$}
\end{overpic}
\label{F:RMT}
\end{center}
\caption{Example of a discrete Riemann mapping}
\end{figure}
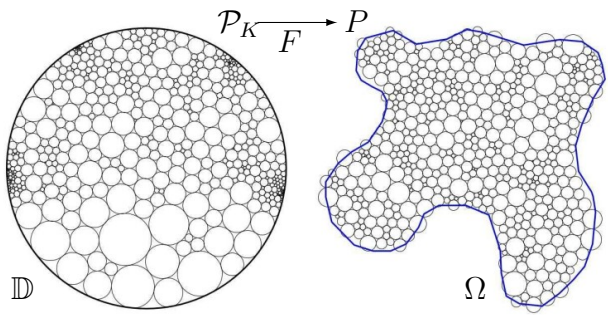

It is the conjecture Thurston made about such discrete conformal
mappings that kicked off the nearly 40 years of development efforts in
circle packing. He suggested that if one were to refine this
construction --- used circle packings $P$ with ever more and smaller
circles --- that the resulting circle packing maps $f$ would converge
uniformly on compact subset of $\bD$ to the classical Riemann Mapping
from $\bD$ onto $\Omega$. Shortly thereafter, this was proven by
B. Rodin and D. Sullivan \cite{RS87} in the case of hexagonal circle
packings. This result has subsequently be expanded to nearly full
generality by many authors; see \cite{kS05} for the story.

The packing $\ccP_K$ has come to be called the {\sl maximal} packing
for $K$ and the mapping $F:\ccP_K\longrightarrow P$ is known as a {\sl
  discrete analytic function}. However, the range of settings has
vastly expanded, so the existence and uniqueness of appropriate
maximal packings has been proven for finite, infinite, and
multiply-connected complexes $K$; discrete analogues are available for
nearly all analytic functions, from entire functions to universal
covering maps, and even branched functions; and the convergence of the
discrete maps to their classical counterparts are established in
nearly every circumstance. One can rightly think of this as
``quantum'' complex analysis --- a discrete theory which not only
mimics the classical, but also converges to it under refinement.

\vspace{10pt} Missing, however, in the pantheon of discrete analytic
functions is the potentially rich family of discrete meromorphic
functions. There is no mystery in the appropriate definition on the
sphere: If $K$ triangulates a sphere and $P$ is a circle packing for
$K$ on $\bP$, then the map $F:\ccP_K\longrightarrow P$
would be a {\sl discrete meromorphic function}. \F{merom}(a) is a non-trivial
example that we will return to in the sequel. Discrete meromorphic functions
can appear more generally as well: \F{merom}(b) represents a discrete
meromorphic function mapping a torus to the sphere.

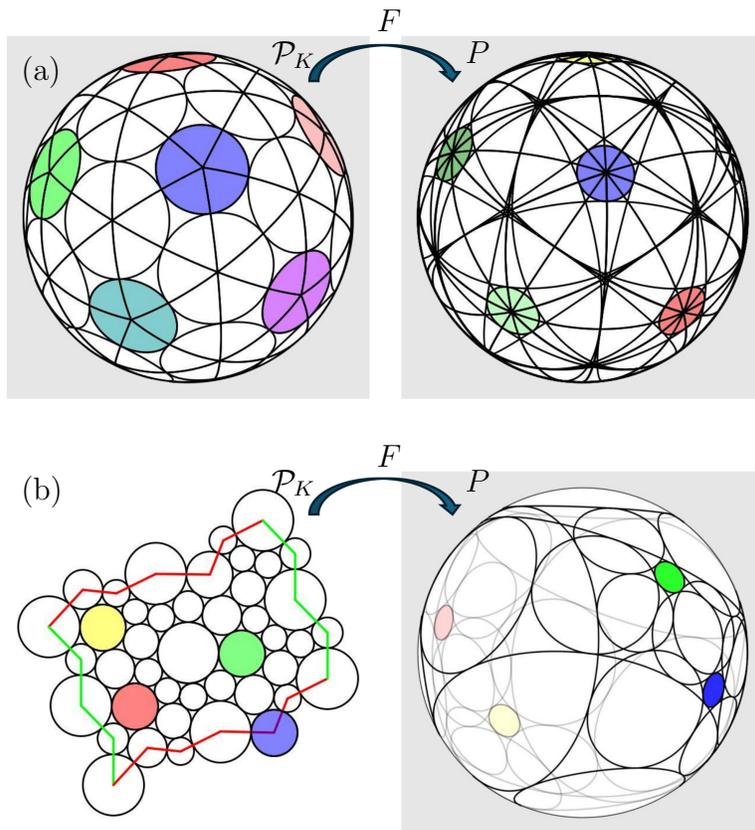
\begin{figure}[h]
\begin{center}
\begin{overpic}[width=.8\textwidth
  ]{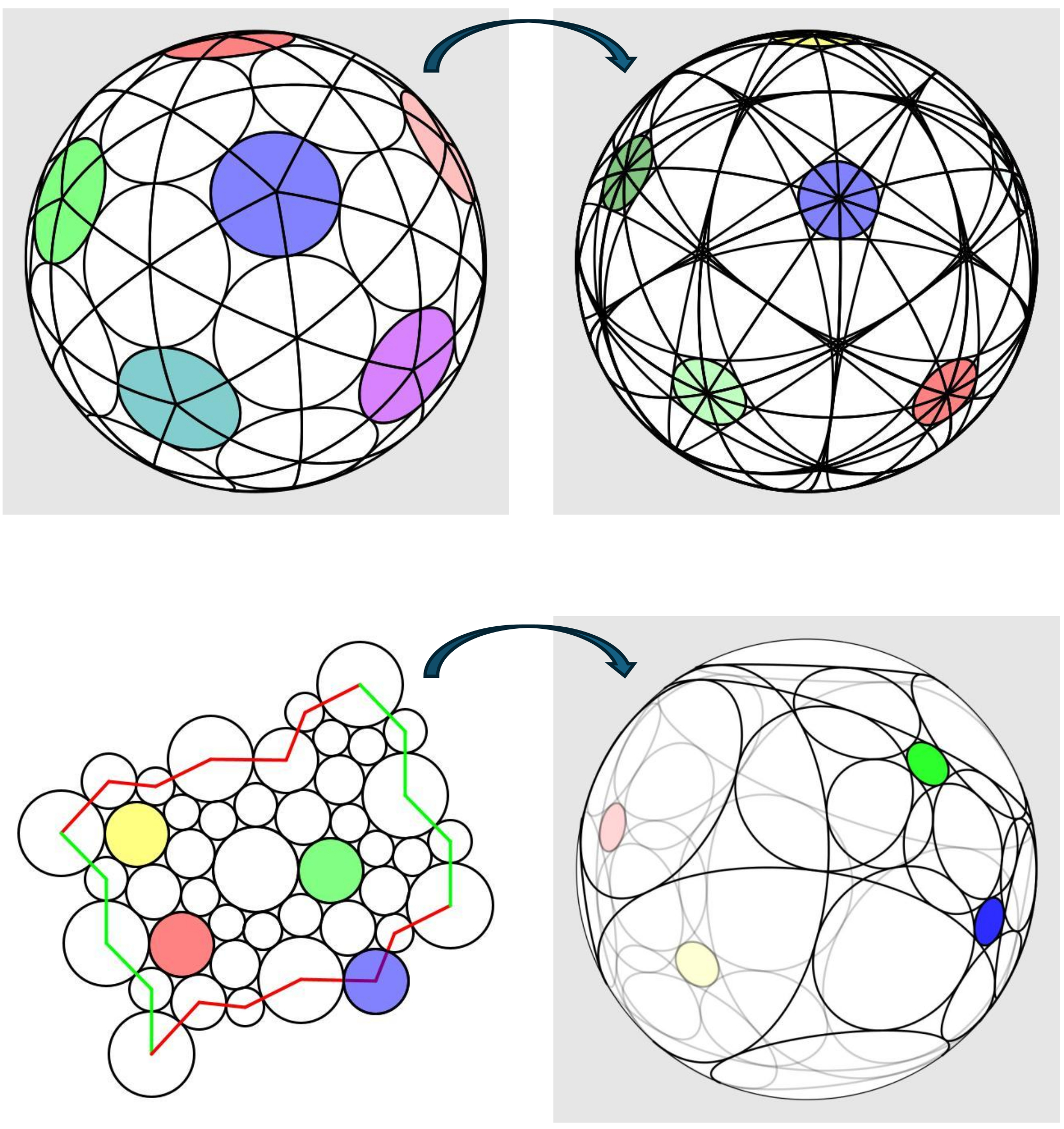}
  \put (2,94) {(a)}
  \put (2,42) {(b)}
  \put (33,96) {$\ccP_K$}
  \put (57,96) {$P$}
  \put (46,100) {$F$}
  \put (33,43) {$\ccP_K$}
  \put (57,43) {$P$}
  \put (46,46) {$F$}
\end{overpic}
\label{F:merom}
\end{center}
\caption{Examples of a discrete meromorphic functions}
\end{figure}

Both these examples owe their existence to
combinatorial symmetries. The first, a discrete analog of the
classical meromorphic function $z^3(3z^5-1)/(z^5+3)$, exploits
dodecahedral symmetry and the special geometry of {\sl Schwarz
  triangles} in $\bP$. There are 12 branched circles, each has 5 petal
circles wrapping twice around it; an isolated flower will be
shown later when we revisit this example. 

The sphere packing of \F{merom}(b), developed jointly with Edward
Crane, is a discrete version of a Weierstrass P-function, mapping a
torus to a 2-sheeted covering of $\bP$ with four simple branch points
(the colored circles).  There is a special symmetry built into its
complex $K$ and the choice of branch vertices, though we have yet to
understand fully why this symmetry ensures a coherent circle packing
in $\bP$.

Absent special symmetries, creating such non-univalent packings is out
of reach, with inherent difficulties in spherical geometry compounded
by the need for branch points.  Methods for constructing packings of
$\bP$, and more generally, packings on Riemann surfaces with
projective structures, are the principal motivation for this work.

\vspace{10pt}
\subsection{On M\"obius Transformations} \label{SS:mob}
Spherical geometry refers here to the geometry of the {\sl Riemann
  sphere} $\bP$, also known as the {\sl complex projective line}.  We
model $\cP$ on the unit sphere centered at the origin in $\bR^3$ and
endowed with a Riemannian metric of constant curvature 1.  The
M\"obius transformations are the members of Aut$(\bP)$, the group of
conformal automorphisms of $\bP$ under composition. These are
intimately connected with both spherical geometry and the geometry of
circles. Here are essential facts to note: $\bullet$ An orientation
preserving homeomorphism $M$ of the sphere maps circles to circles if
and only if $M$ is a M\"obius transformation. $\bullet$ In particular,
if $P$ is a circle packing in $\bP$, then $M(P)$ is a circle packing
in $\bP$. $\bullet$ If $\{C_1,C_2,C_3\}$ and $\{c_1,c_2,c_3\}$ are any
two triples of mutually tangent circles, then there exists a unique
M\"obius transformation $M$ so that $M(C_j)=c_j,\, j=1,2,3$.
$\bullet$ The conformal automorphisms of the unit disc $\bD$ and of
the complex plane $\bC$, Aut$(\bD)$ and Aut$(\bC)$, are subgroups of
Aut$(\bP)$.  $\bullet$ Aside from the identity $\bI$, M\"obius
transformations all have 1 or 2 fixed points and fall into one of
three categories, {\sl parabolic}, {\sl elliptic}, or {\sl
  hyperbolic}; the parabolic are those with a single fixed point.

It is routine to represent a M\"obius transformation $M$ in complex
arithmetric as a {\sl linear fractional transformation}
$M(z)=(az+b)/(cz+d)$, where $a,b,c,d$ are complex coefficients with
$ad-bc\not=0$. Computationally, we will work with these in the form of
$2\times 2$ complex matrices:
\begin{align} 
  &M(z)=(az+b)/(cz+d)\ \text{ is represented by } \notag \\
  &M=\begin{bmatrix}a & b\\ c & d\end{bmatrix},\
  \text{ with det}(M)=ad-bc\not=0. \notag
\end{align}
Composition of M\"obius transformations is represented by normal
matrix multiplication of their matrices, and the inverse of a
transformation is represented by the inverse of it matrix. The matrix
representation $M$ may be multiplied by any non-zero complex scalar,
so we will often normalize to ensure that $ad-bc=1$. Furthermore, if
$M$ is parabolic, we can ensure that trace$(M)=2$.

\vspace{10pt}
\subsection{On Software} \label{SS:cp}
Computations, visualizations, and experiments have been drivers of
circle packing since the topic's inception, principally due to the
algorithm Thurston introduced in his 1985 talk. The many refinements
of his algorithm now allow the computation of impressively large and
complicated complexes, some with millions of circles. See, for
example, the algorithms in \cite{CS03}, \cite{COS16},
\cite{COS17}. These capabilities and connections to conformal geometry
have in turn allowed significant applications of circle packing in
mathematics \cite{BS17}, in brain imaging \cite{mHS04}, physics
\cite{vH19}, engineering \cite{KS18}, not to mention art and
architecture.

It is especially important to note the key role that open-ended
experiments, visualizations, and serendipity play, even in the purely
theoretical aspects of circle packing. The topics in this paper are
just the latest examples.  Experiments require a laboratory, and for
the work here that laboratory is the open source Java software package
{\tt CirclePack}, available on {\sl Github} \cite{kS92}. All images
in this paper and the computations behind them due to {\tt CirclePack}.
Moreover, scripts are available from the author to repeat and extend
the experiments.

\vspace{20pt}
\section{Classical, Discrete, and Intrinsic} \label{S:class}
The Schwarzian derivative was discovered by Lagrange and named after
H. Schwarz by Cayley. It is a fundamental M\"obius invariant in
classical complex analysis, with important applications in topics from
function theory, differential equations, and Teichm\"uler theory,
among others. Suppose $\phi:\Omega\mapsto \Omega'$ is an analytic
function between domains $\Omega,\Omega'$ of the complex plane 
whose derivative $\phi'$ does not vanish.  The
{\bf\it Schwarzian derivative} $S_{\phi}$ is defined by
\begin{equation}
  S_{\phi}(z)=
  \dfrac{\phi'''(z)}{\phi'(z)}-\dfrac32(\dfrac{\phi''(z)}{\phi'(z)})^2.
\end{equation}
There is also a useful {\bf\it pre-Schwarzian derivative} $s_{\phi}$:
\begin{equation}
  s_{\phi}(z)=(\ln(\phi'(z)))'=\dfrac{\phi''(z)}{\phi'(z)}
  \ \ \ \Longrightarrow\ \ \ S_{\phi}(z)=
  s_{\phi}'(z)-\frac12(s_{\phi}(z))^2.
\end{equation}

\vspace{5pt}
\noindent The Schwarzian derivative is valuable because of its
intimate association with M\"obius transformations. By direct
computation, if $m(z)$ is a M\"obius transformation,
then $S_{m}\equiv 0$. The converse also holds: if $S_{\phi}\equiv 0$
in $\Omega$, then $\phi$ is a M\"obius transformation. In general
terms, then, {\sl the Schwarzian derivative of a function indicates
  how far that function differs from being M\"obius.} Reinforcing this
intuition is the fact that the Schwarzian derivative is invariant
under post-composition with M\"obius transformations:
$S_{m\circ\phi}\equiv S_{\phi}$. Moreover, for pre-composition, the
chain rule gives
\begin{equation} \label{E:chain}
  S_{\phi\circ m}(z)=S_{\phi}(m(z))\cdot (m'(z))^2.
\end{equation}

These features motivate development of our discrete Schwarzian
derivative. This began with work of Gerald Orick in his PhD thesis
\cite{gO10}.  He was searching for a discrete analogue of a classical
univalence criterion due to Nehari. Suppose $\phi$ is an analytic
function on the unit disc $\bD$. If it were M\"obius, then, of course,
it would be univalent (i.e., injective). Nehari proved that if $\phi$
is close enough to being M\"obius, in the sense $|S_{\phi}(z)|\le
2/(1-|z|^2)^2,\forall z\in \bD$, then $\phi$ is univalent.  Though the
search for a discrete version of Nehari's result continues, Orick laid
the groundwork for our notion of Schwarzian derivative, thereby
opening a rich vein of questions. (Discretized Schwarzian derivatives
have appeared via {\sl cross-ratios} for circle packings with regular
square grid or hexagonal combinatorics (see \cite{oS97} and
\cite{HS98}, respectively) and in circle pattern literature (see
\cite{wyL24} for example.)

\subsection{Patches} \label{SS:patch}
In concert with the notion of a patch in defining classical conformal
structures, a ``patch'' in a circle packing $P$ will refer to the four
circles forming a pair of contiguous faces. Our terminology will be
used in both combinatorial and geometric senses. Thus we will write
$\fp=\{v,w\,|\,a,b\}$ for the combinatorial patch formed by
faces $f=\{v,w,a\}$ and $g=\{w,v,b\}$ in the complex $K$. We might
also use the notation $\fp=\{f\,|\,g\}$. The common edge of
the faces is $e=\{v,w\}$, and by convention is positively oriented
with respect to the interior of $f$.

The circles of $P$ impose a geometry on $K$, and the corresponding
geometric patch in $P$ is $\fp=\{C_v,C_w\,|\,C_a,C_b\}$
forming faces $f$ and $g$ based on the triples $\{C_v,C_w,C_a\}$ and
$\{C_w,C_v,C_b\}$, respectively, and with common edge $e=\{C_v,C_w\}$.

Parallel to the classical setting we will also be working with a
discrete analytic function $F:P\mapsto P'$ mapping $P$ to a second
circle packing $P'$ sharing the complex $K$.  For the patch
$\fp=\{C_v,C_w\,|\,C_a,C_b\}$ of $P$ we have the corresponding
patch $\fp'=F(\fp)=\{C_v',C_w'\,|\,C_a',C_b'\}$ of
$P'$, and corresponding geodesic triangles $f',g'$ with shared edge
$e'$ of $P'$.

The {\it discrete Schwarzian derivative} of $F$, denoted $\Sigma_F$, will
be a complex function defined on the collection of interior edges
$e=\{v,w\}$ of the domain packing $P$. More concretely, the value
$\Sigma_F(e)$ will be associated with the tangency point $t_e$ of $C_v$ and
$C_w$.

Fix attention on a combinatorial patch $\fp=\{v,w\,|\,a,b\}$
in $K$ with faces $f,g$ in $P$ and $f',g'$ in $P'$ and directed edge
$e$. There exist M\"obius transformations $m_f$ and $m_g$ identifying
corresponding faces. We write
\begin{equation}\label{E:mfmg}
  m_f(f)=f'\quad\text{ and }\quad m_g(g)=g'.
  \end{equation}
A brief note about these equalities: There is a unique M\"obius
transformation taking the tangent triple $\{C_v,C_w,C_a\}$ to the
corresponding triple $\{C_v',C_w',C_a'\}$; in practice, it is found by
mapping the three tangency points of one to the corresponding tangency
points of the other. In hyperbolic and euclidean settings, these
M\"obius maps are homeomorphisms of the geodesic triangles formed by
the centers of the triples.  In spherical geometry, however, M\"obius
transformations do not necessarily preserve geodesics and circle
centers, so the equalities of (\ref{E:mfmg}) are symbolic rather that
literal.

\begin{figure}[h]
\begin{center}
\begin{overpic}[width=.8\textwidth
  ]{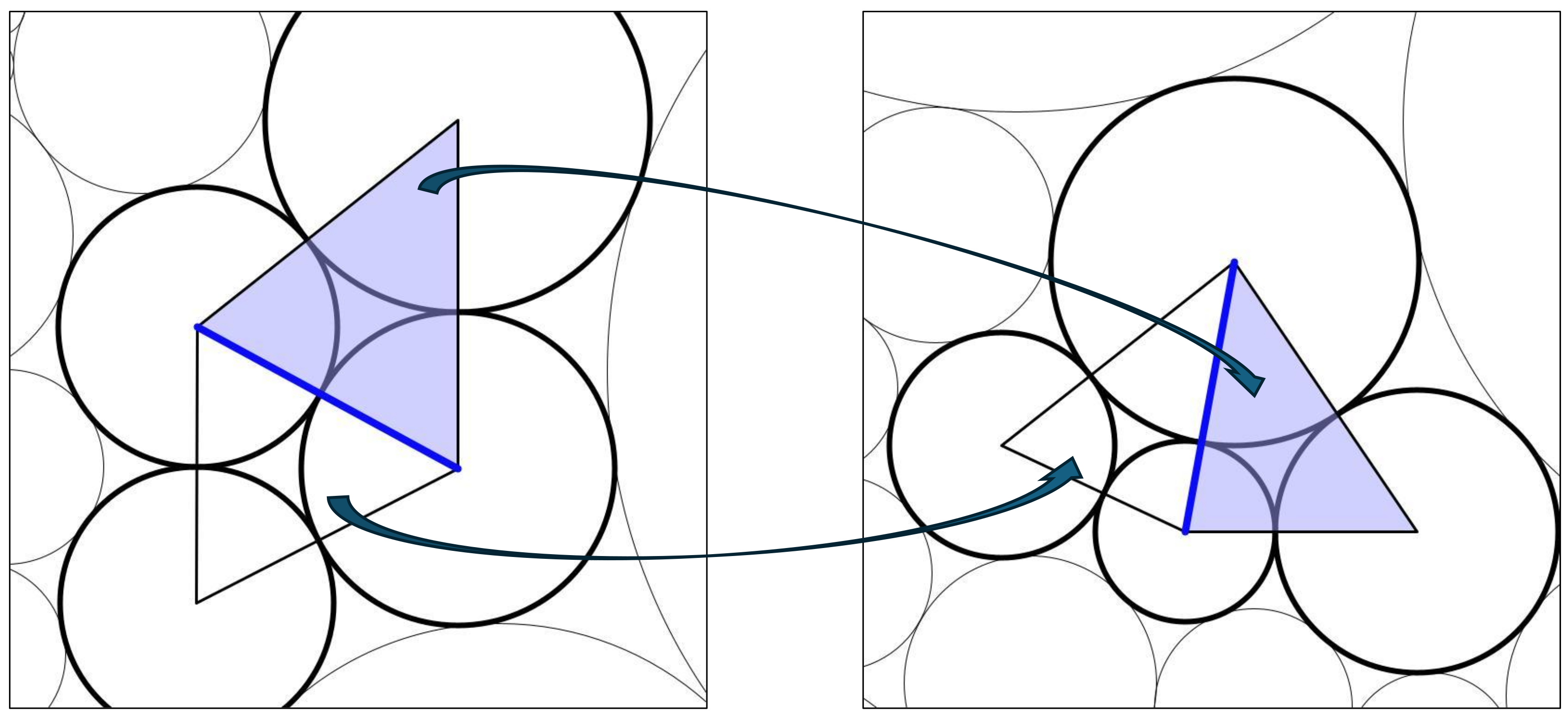}
  \put (23.5,30) {$f$}
  \put (81,19) {$f'$}
  \put (20.5,16.5) {$g$}
  \put (67.5,18) {$g'$}
  \put (48,8) {$m_g$}
  \put (48,34) {$m_f$}
  \put (5,40) {$P$}
  \put (94,40) {$P'$}
  \put (9,27) {$C_v$}
  \put (30.5,14) {$C_w$}
  \put (78,31.5) {$C'_v$}
  \put (74,8) {$C'_w$}
\end{overpic}
\caption{The discrete Schwarzian derivative for an edge}
\label{F:sd_def}
\end{center}
\end{figure}
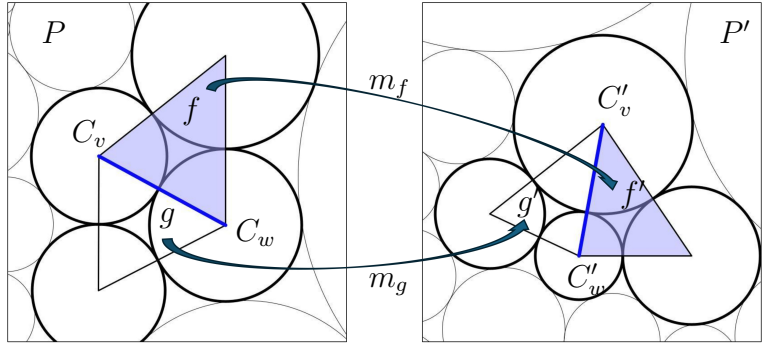

\vspace{10pt} With $m_f$ and $m_g$ we may now define $M_F(e)$ as the
M\"obius transformation
\begin{equation}M_F(e)=m_g^{-1}\circ m_f.
\end{equation}
Though we have adjusted notation slightly, $M_F(e)$ is the {\it
  (directed M\"obius) edge derivative} of Orick.  The maps $M_F(e)$
have very particular forms observed by Orick. In particular, the
definition of $M_F(e)$ shows that it fixes $t_e$, the tangency point
of $C_v$ and $C_w$, and that it fixes $C_v$ and $C_w$ themselves as
points sets. As a result, if $M_F(e)$ is not the identity then it is
necessarily a parabolic M\"obius transformation.  We are free to
normalize so that trace$(M_F(e))=2$ and det$(M_F(e))=1$.  In this
case,
\begin{equation} \label{E:MF}
  M_F(e)=\bI + \sigma\cdot 
\begin{bmatrix}t_e & -t_e^2\\
  1 & -t_e
\end{bmatrix}.
\end{equation}
Moreover, if $\eta=e^{i\theta}$ is the common tangent to $C_v$ and
$C_w$ at $t_e$ and is pointing outward from face $f$, then $\sigma$ is
a real multiple of its complex conjugate, $\overline{\eta}$.

\begin{definition} \label{D:dsd}
  Let $F:P\longrightarrow P'$ be a discrete analytic function. For each
  interior edge $e$ of $P$, the value $\sigma$ arising in the
  computation of $M_F(e)$ as described above is defined as the {\bf\it
    (discrete) Schwarzian derivative} of $F$ on the edge $e$ and we
  write $\sigma=\Sigma_F(e)$.
\end{definition}

There are several properties to observe here:

\vspace{5pt}
\begin{enumerate}
\item[$\bullet$]{Suppose $-e$ denotes the edge $e$ but oppositely
  oriented. Then $M_F(-e)=(M_F(e))^{-1}$, implying
  $\Sigma_F(-e)=-\Sigma_F(e)$.}

\vspace{5pt}
\item[$\bullet$]{$\Sigma_F(e)=0$ if and only if $M_F(e)$ is the
  identity.}

\vspace{5pt}
\item[$\bullet$]{If $F$ itself were M\"obius, then $m_f\equiv
  m_g\equiv F$ and $M_F(e)=\bI$ for every interior edge
  $e$. Conversely, if $M_F(e)=\bI$ for every interior edge $e$, then a
  simple face-to-face continuation argument would show that $F$ itself
  is M\"obius, with $m_f\equiv F$ for every face $f$.}

\vspace{5pt}
\item[$\bullet$]{Suppose we follow $F$ by a M\"obius transformation
  $m$, say $G\equiv m\circ F: P\rightarrow P''=m(P')$. The new face
  M\"obius maps for $\fp$ are $\widehat{m}_f:f\rightarrow
  m(f')=f''$ and $\widehat{m}_g:g\rightarrow m(g')=g''$. Since
  $\widehat{m}_f=m\circ m_f$ and $\widehat{m}_g=m\circ m_g$, note that
  \begin{equation} \notag\label{E:MG}
    M_G(e)=\widehat{m}_g^{-1}\circ \widehat{m}_f= m_g^{-1}\circ m^{-1}\circ
    m\circ m_f=M_F(e).
  \end{equation}
  That is, the operator $M$ and consequently $\Sigma_F$
  are M\"obius invariant.}

\vspace{5pt}
\item[$\bullet$]{Computations in the Appendix show that the discrete
  chain rule under pre-composition by a M\"obius transoformation
  diverges slightly from the classical rule of (\ref{E:chain}); see
  (\ref{E:SSigma}).}

\item[$\bullet$]{It is very likely that if a sequence $\{F_n\}$ of
  discrete analytic functions converges on compacta to a classical
  analytic function $\phi$, then the sequence $\{\Sigma_{F_n}\}$ also
  converges on compacta to $S_{\phi}$. Results of Z-X. He and Oded
  Schramm in \cite{HS98} can be used to confirm this for packings with
  hexagonal combinatorics, but it remains open for more general circle
  packings.}
  \end{enumerate}

\subsection{Intrinsic Schwarzians} \label{SS:intrinsic}
The Schwarzian derivative is associated with mappings {\sl between} circle
packings. However, we can exploit the same notion in a
local sense to provide an ``intrinsic schwarzian'' for each interior
edge of an individual packing. For this we need only consider a
target patch $\fp$ (as occurs within a packing $P$, for instance) and a
standard {\bf\it base patch} $\fp_{\Delta}$ which we describe
next.

The base patch $\fp_{\Delta}$ consists of two contiguous equilateral
triangles, $f_{\Delta}$ and $g_{\Delta}$. Here $f_{\Delta}$ is formed
by the tangent triple of circles of radius $\sqrt{3}$, symmetric about
the origin, and having distinguished edge $e_{\Delta}$ running
vertically through $z=1$.  Note that the incircle of $f_{\Delta}$
intersects its edges at their points of tangency and these are the 3rd
roots of unity; thus the tangency point for $e_{\Delta}$ is $t_e=1$
and the outward unit vector is $\eta=1$. The face $g_{\Delta}$ is an
equilateral triangle contiguous along $e_{\Delta}$, so it shares the
two circles of $e_{\Delta}$ and its third circle is centered at $x=4$.

A target patch $\fp$ is formed by two triples of circles
sharing a pair of circles and has triangular faces we will denote by
$f$ and $g$. We do our computation for edge $e_{\Delta}$ as usual, as
though for the mapping $F:\fp_{\Delta}\rightarrow
\fp$.  That is, we compute the face M\"obius transformations
$m_f,m_g$, so $m_f(f_{\Delta})=f$ and $m_g(g_{\Delta})=g$, and then
M\"obius $M(e_{\Delta})=m_g^{-1}\circ m_f$.  Referring to
(\ref{E:MF}), note that $\eta=1$, so $\sigma$ is some real value $s$,
and $t_e=1$, implying
\begin{equation} \label{E:mf}
  M(e_{\Delta})=\bI+s\cdot
  \begin{bmatrix}1 & -1\\
    1 & -1
  \end{bmatrix}=
  \begin{bmatrix} 1+s & -s\\
    s & 1-s
  \end{bmatrix},
\end{equation}
This matrix is important in the following, so we use the notation
$M_s$. 

\begin{definition} \label{D:intsch}
Given a geometric patch $\fp$, the $(1,2)$ entry of the
M\"obius transformation $M_s$ described in (\ref{E:mf}) is a real
value $s$ and is defined as the {\bf\it (intrinsic) schwarzian} for the
shared edge $e$ of $\fp=\{f\,|\,g\}$.
\end{definition}

The intrinsic schwarzian $s$ completely characterizes the target patch
$\fp$ up to M\"obius transformations. In particular, it is unchanged
if we interchange the labels $f$ and $g$; it is unchanged if $\fp$ is
replaced by $m(\fp)$ for a M\"obius transformation $m$; and if $\fp_1$
and $\fp_2$ are two geometric patches with identical intrinsic
schwarzians, then there exists a M\"obius transformation $m$ so that
$\fp_1=m(\fp_2)$.

Computations in the Appendix establish the connection between
Schwarzian derivatives and intrinsic schwarzians.  Given
$F:P\longrightarrow P'$, consider a patch $\fp\subset P$, its image
patch $\fp'\subset P'$, their edges $e,e'$, respectively, and the
Schwarzian derivative $\sigma=\Sigma_F(e)$. Let $s$ and $s'$ denote
the intrinsic schwarzians for $e$ and $e'$, respectively.  Let $m$ be
the M\"obius transformation of the base face $f_{\Delta}$ onto the
face $f$ of $\fp$. Computations in the Appendix show
\begin{equation} \label{E:ssigs}
  s'=s+\Sigma_F(e)\cdot m'(1)=s+\dfrac{\Sigma_F(e)}{(c+d)^2},
  \text{ where }
  m(z)=\dfrac{az+b}{cz+d},\ ad-bc=1.
\end{equation}
As a side note, observe that $\Sigma_F(e)\cdot m'(1)$ is real.

\vspace{10pt} We finish this subsection by illustrating schwarzians in
relation to the base patch $\fp_{\Delta}$. This not only lets the
reader gain some intuition, but also leads us to the computational
machinery central to the
remainder of the paper.

\vspace{15pt}
\begin{figure}[h]
\begin{center}
\begin{overpic}[width=.96\textwidth
  ]{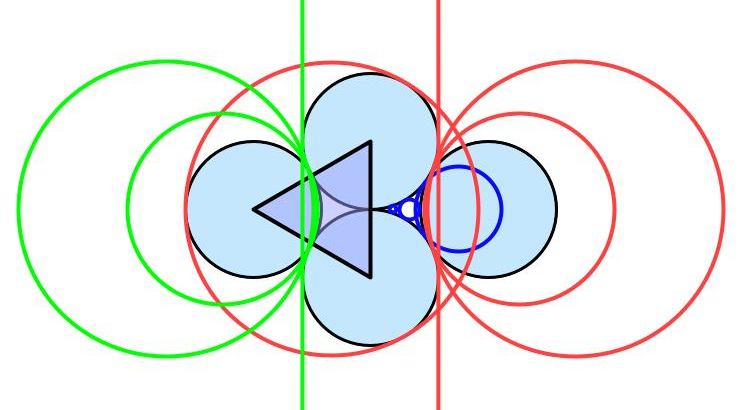}
\put (50.8,28.75) {1}
\put (44,12) {$C_v$}
\put (44,39) {$C_w$}
\put (26.7,26) {$C_a$}
\put (70,26) {$C_b$}
\put (46,3) {$s=-.6$}
\put (49,4.65) {\vector(.221,.8){5.6}}
\put (70,3) {$s=0.0$}
\put (73,4.65) {\vector(-0.22,.8){3.6}}
\put (88,3) {$s=0.1$}
\put (91,4.65) {\vector(-1,2){2.8}}
\put (20,3) {$s=1.0$}
\put (23,4.65) {\vector(1,2){6.8}}
\put (3,3) {$s=.95$}
\put (5.6,4.65) {\vector(1,2.75){2.75}}
\put (43,52.5) {$s=1-\frac{1}{\sqrt{3}}$}
\put (46,52) {\vector(0.1,-0.9){.55}}
\put (68,50) {$s=(1-\frac{1}{\sqrt{3}})/2$}
\put (67,50.25) {\vector(-1.0,-0.2){7}}
\put (10,50) {$s=(1+\frac{1}{\sqrt{3}})/2$}
\put (29.5,50.2) {\vector(1.0,-0.18){10}}
\put (44.75,30) {$f_{\Delta}$}
\end{overpic}
\end{center}
\caption{Sampling some intrinsic schwarzians}
\label{F:base}
\end{figure}
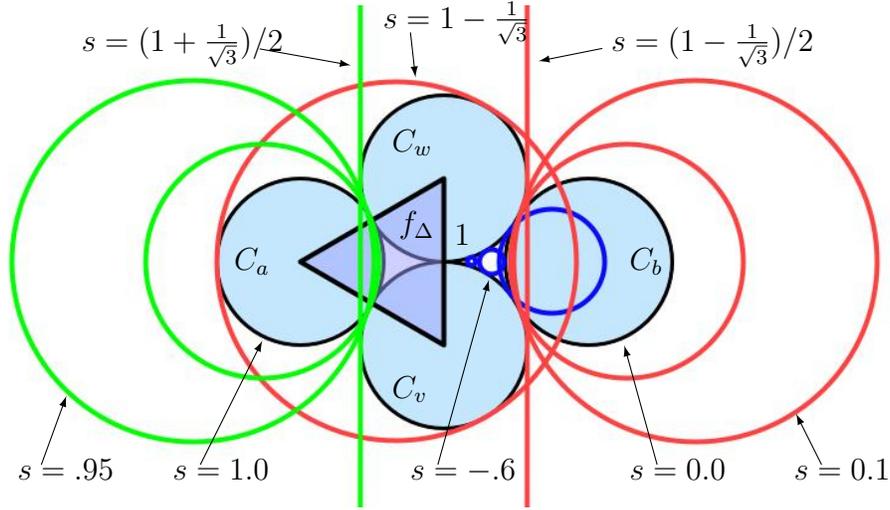

Our base patch $\fp_{\Delta}$ appears in \F{base} as the four light blue
discs, with $C_b$ the one on the right.
Replacing $C_b$ with some new circle $\widehat{C}$ which is tangent
to $C_v$ and $C_w$, forms a new patch. Each new patch leads to some
schwarzian $s$ for the edge $e$, so we will denote that new circle
by $\widehat{C}_s$. Indeed,
by our work above, we have $\widehat{C}_s=M_s^{-1}(C_b)$.

To explore the various situations, consider these 5 specific values of $s$:
\begin{equation}\notag
  s_0=0\ <\ s_1=\dfrac{1-1/\sqrt{3}}{2}\ <\ 
  s_2=1-1/\sqrt{3}\ <\ s_4=\dfrac{1+1/\sqrt{3}}{2}\ <\ s_5=1
\end{equation}

When $s=s_0=0$, $\widehat{C}_s$ is identical to $C_b$. As $s$
decreases through negative values, the corresponding circles $\widehat{C}_s$ get
smaller as they drop into the right crevasse between $C_v$ and $C_w$, as
shown with some blue examples in \F{base}. As $s$ increases from $s_0$
to $s_2$, the spherical radius of $\widehat{C}_s$ increases as shown
by the red circles. Along the way, when $s=s_1$ then $\widehat{C}_s$
is the line through $\infty$, the vertical line tangent to $C_v$ on
$C_w$ on their right.  Reaching $s=s_2$, the circle $\widehat{C}_s$ is
suddenly tangent to all three of the circles forming $f_{\Delta}$, but
with $\infty$ in its interior. For $s$ larger than $s_2$, the
spherical radius of $\widehat{C}_s$ is decreasing, as shown with green
examples. The $\widehat{C}_s$ now overlap $C_a$, and on reaching
$s_4$, $\widehat{C}_s$ is the vertical line tangent to $C_v$ and $C_w$
on their left. As $s$ grows beyond $s_4$, the circles are moving more
deeply into the left crevasse between $C_v$ and $C_w$.  On reaching
$s=s_5$, $\widehat{C}_s$ is identical to $C_a$. This is a critical
juncture: the interstices for faces $f$ and $g$ are now reflections of
one another through the unit circle. If we let $s$ exceed $s_5=1$, then
these interstices overlap, a condition we will exclude in later work
on branched flowers.

\vspace{10pt} Now move to consideration of a generic patch
$\fp=\{v,w\,|\,a,b\}$. Suppose the centers and radii for
circles $\{c_v,c_w,c_a\}$ forming $f$ and the intrinsic
schwarzian for the edge $e=\{c_v,c_w\}$ are known. Then one can compute
the unknown circle $c_b$, and consequently fix the face
$g=\{c_w,c_v,c_b\}$. Here are the details:

There exists a M\"obius transformation $m_f$ mapping $f_{\Delta}$ to
$f$. Therefore, $m_f^{-1}(f)=f_{\Delta}$, and so the patch
$\widehat{\fp}=m_f^{-1}(\fp)$ will be analogous to those depicted in
\F{base}. Because schwarzians are invariant under M\"obius
transformations, the schwarzian for $\widehat{\fp}$ will again be $s$,
meaning that its circle $m_f^{-1}(c_b)$ must be $\widehat{C}_s$. Noting
that $\widehat{C}_s=M_s(C_b)$, the unknown circle $c_b$ of
$\fp$ is given by
\begin{equation} \label{E:cb}
  c_b=(m_f\circ M_s^{-1})(C_b)\ \ \text{ where }\ \ M_s^{-1}=
  \begin{bmatrix}
    1-s & s\\
    -s & 1+s
  \end{bmatrix},
\end{equation}
a fact we will use extensively in the sequel.

\section{Packing Layouts} \label{S:layout}
Construction of a circle packing for a given complex $K$ typically
starts (as Thurston did) with the computation of a {\sl packing label}
$R=\{R(v):v\in K\}$ containing the circle radii. Then comes the {\sl
  layout} process, the computation of the circle centers. This process
utilizes a spanning tree $T$ chosen from the dual graph of $K$. Any
face $f_0$ of $T$ may be designated as the root. Using the radii of
its three vertices, one can lay out a tangent triple of circles
forming the geometric face $f_0$. For each dual edge $\{f,g\}\in T$,
if face $f$ is in place, then two of its circles are shared with $g$,
and the radius of the remaining circle of $g$ is enough to compute its
unique position.  Thus, starting with the geometric root face $f_0$,
one can proceed through $T$ to lay out all remaining circles,
resulting in the final packing $P$.

This process could instead be carried out using schwarzians, if they
were available.  Write $S=\{s(e): e\ \text{interior edge}\}$ for the
(intrinsic) schwarzians of interior edges for some packing
$P$. Starting with {\sl any} (!) tangent triple of circles and
identifying it as the base face $f_0$, one can again progress through
the edges $\{f,g\}$ of the dual spanning tree $T$. If the face $f$ is
in place, then using a patch $\fp=\{f\,|\,g\}$ and the schwarzian
$S(e)$ for its shared edge $e$, one can apply (\ref{E:cb}) to
determine the radius and center of the third circle of $g$.
Progressing thus through $T$ yields a final packing $P$ for $K$.
Since the whole of $P$ is determined by the initial geometric face
$f_0$;, we can obtain any M\"obius image $m(P)$ by starting the layout
with the appropriate base face.

There are some issues to address: Using the traditional layout
approach {\sl via} radii, the geometry of $P$ must be that of the
given label $R$. The layout approach {\sl via} intrinsic schwarzians,
on the other hand, is by its very nature carried out on the
sphere. Indeed, whether the final packing $P$ lives in the plane or
the hyperbolic plane might well be dictated by the choice of the
initial face $f_0$.  Perhaps this is the advantage of using
schwarzians: one can lay out packings on the sphere or, more
generally, on surfaces with projective structures.

I would also point out that when $K$ is not simply connected, the
layout process, whether with radii or schwarzians, is more subtle;
laying out a closed chain of faces which is not null homotopic can
lead to non-trivial holonomies, meaning the data is not associated
with a circle packing. Let us therefore stick to simply connected
complexes $K$ for now.

\subsection{The Difficulty} \label{SS:dculty}
The difficulty in the schwarzian approach lies not with layout, but
rather with computation of the data itself. In introducing circle
packing to the world, Thurston also graced us with an iterative
algorithm for computing radius data. With radii in hand, one can
easily lay out the circles to form $P$. However, his algorithm is
restricted to the euclidean and hyperbolic settings, and despite
considerable effort, there is no known algorithm in spherical
geometry. There are two key ingredients in Thurston's
clever algorithm:

\vspace{10pt}
\noindent{\narrower{
    
\noindent$\bullet$\ {\bf Criteria:} Given a label $R$ of putative
radii, one can directly compute the set $\{\theta_R(v): v\in
K\}$ of associated {\sl angle sums} at the vertices of
$K$. $R$ is a packing label if and only if $\theta_R(v)$ is an integer
multiple of $2\pi$ for every interior $v$.

\vspace{10pt}
\noindent$\bullet$\ {\bf Monotonicities:} There are simple
monotonicities in the effects that adjusments in radius labels have on
associated angle sums. In particular, a packing label is the zero set
of a convex functional, guaranteeing existence and uniqueness (and
computability) of solutions.

}}

\vspace{10pt} It is the monotonicity that fails us in the spherical
setting. Building new computational capabilities is the main
motivation for looking at schwarzians.  That is, we want to replace
the data provided by a vertex label $R=\{R(v):v\in K\}$ with
that of an intrinsic schwarzian {\sl edge} label $S=\{S(e):e\in
K\}$.

\begin{definition} \label{D:Slabel}
  Let $K$ be a simply connected complex and let $S$ be an edge label,
  that is, a set of real numbers, one for each interior edge of $K$.
  We call $S$ a {\bf\it packing (edge) label} if there exists a circle
  packing $P$ on the Riemann sphere whose intrinsic schwarzians are
  given by $S$.
\end{definition}

The main question: {\sl What are the packing labels?} Based on
experience with radius labels, and in particular on results in
\cite{BSW11}, we anticipate that the packing labels will form a
$(p-3)$-dimensional differentiable subvariety $\fS\subset\bR^k$, where
$p$ and $k$ are the numbers of interior vertices and interior edges of
$K$, respectively. Describing $\fS$ and more importantly, computing
specific packing labels, appears to be very challenging.  Our modest
approach has been to set up mechanisms for experimentation and
discovery. Observations from experiments in {\tt CirclePack} have led
to the clunky but serviceable Theorems \ref{T:rational} and
\ref{T:univalent} below on packing labels for flowers. As for
constructing packing labels for whole complexes, I am less sanguine.
Even working with radii data, monotonicity may fail in our spherical
setting, and without monotonicity, methods for generating and
manipulating packing edge labels will require major new insights and
numerical machinery. Edward Crane has built an explicit example of a
complex $K$ triangulating the sphere with a designated set of its
vertices as branch points which has two M\"obius inequivalent
realizations as circle packings of $\bP$. Non-uniqueness is a sobering
feature when looking for an algorithm. Nonetheless, let's do what we
can and begin by looking at individual flowers.

\section{Flowers} \label{S:flowers}

The search for general packing critera naturally begins with the
study of packing labels for individual flowers. One can easily
generate randomized $n$-flowers for any $n$, and
thereby obtain a wealth of associate schwarzian labels.
Our work, however, lies in the reverse direction: given a label
$\{s_0,s_2,\cdots,s_{n-1}\}$, how can
one tell if it is a packing label? 

Here we develop and exploit a general process for laying out flowers
in a normalized setting.  This has been implemented in {\tt
  CirclePack}, and our investigations have relied on the flexible
nature of its computations and visualizations. Our interest lies with
closed flowers and after preliminaries we work in successive subsections
on un-branched flowers, univalent flowers, and finally on branched flowers.

\subsection{Notation and Preliminaries}
In a tangency circle packing, the {\sl flower} of the
circle $C=C_v$ for an interior vertex is denoted
$\{C;c_0,c_1,\cdots,c_{n-1}\}$, where $c_0,\cdots,c_{n-1}$ is the chain of
{\sl petals} which wrap around $C$ with the last tangent to the
first. The ordered chain of interior edges emanating from $v$ may be
written as $\{e_0,e_1,\cdots,e_{n-1}\}$, where $e_j$ is the edge
$\{C,c_j\}$ and hence is the shared edge in the patch
$\fp=\{c_j,C\,|\,c_{j-1},c_{j+1}\}$. We write $\{s_0,s_1,\cdots,s_{n-1}\}$
for the corresponding intrinsic schwarzians, a packing label for this
flower. (Note that the indexing for $n$-flowers is always modulo $n$.)

Putting the first triple of circles, $\{C,c_{n-1},c_0\}$, in place, one
can then use schwarzians $s_0,s_1,\cdots,s_{n-3}$ in succession to
place $c_1,c_2,\cdots,c_{n-2}$ and possibly complete the geometric
flower. But is this a flower? Some conditions must be needed since
this procedure did not even utilize the given schwarzians $s_{n-2}$ or
$s_{n-1}$.  To be a packing label, the layout must avoid three potential
incompatibilities:

\vspace{10pt}
\noindent{\narrower {
    
    \noindent{\bf (a)} The flower may fail to close; that is,
    $c_{n-1}$ may fail to be tangent to $c_0$.

    \vspace{5pt}
    \noindent{\bf (b)} The patch
    $\fp_{n-1}=\{c_{n-1},C\,|\,c_{n-2},c_0\}$ may fail to
    have schwarzian $s_{n-1}$.

    \vspace{5pt}
    \noindent{\bf (c)} The patch $\fp_0=\{c_0,C\,|\,c_{n-1},c_1\}$
    may fail to have schwarzian $s_0$.

}}

\vspace{10pt} Since flowers and their schwarzians are unchanged under
M\"obius transformations, one can choose any convenient
normalization. We have chosen that illustrated in \F{f7} (in this
instance, a 7-flower). This approach sidesteps the numerical
difficulties working with infinity.  The notations of the figure are
those we will use throughout the paper: the upper half plane
respresents the central circle $C$, the half plane $\{z=x+iy:y\le
-2\}$ represents $c_0$ (so the tangency between $C$ and $c_0$ lies at
infinity), and the petal $c_1$ of radius 1 is tangent to $C$ at
$t_1=0$. The successive petals $\{c_2,\cdots,c_{n-1}\}$ have tangency
points $\{t_2,\cdots,t_{n-1}\}$. The successive {\sl euclidean radii} will
be denoted by $\{r_2\cdots,r_{n-1}\}$, and successive {\sl
  displacements} by $\delta_j= t_{j+1}-t_j$.

\begin{figure}[h] 
\begin{center}
\begin{overpic}[width=.9\textwidth
  ]{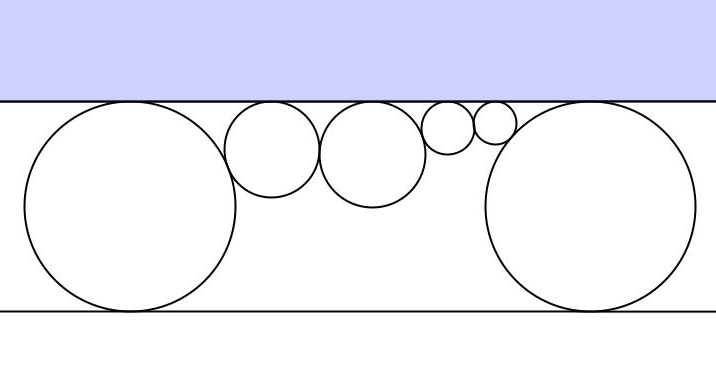}
  \put (56,46) {$\delta_3$}
  \put (52.1,45) {\vector(1,0){10.2}}
\put (6,48) {$C$}
\put (6,4) {$c_0$}
\put (15,6.3) {$-2\,i$}
\put (11.5,15.5) {$c_1$}
\put (34.5,29.5) {$c_2$}
\put (77,15.5) {$c_6$}
\put (40,33) {$r_2$}
\put (38,33.5) {\vector(1,-1){5}}
\put (85.5,21) {$r_6:=1$}
\put (82,26) {\vector(1,-1.75){7.8}}
\put (13,42) {$t_1=0$}
\put (36,42) {$t_2$}
\put (48.6,42) {$t_3$}
\put (51.8,40) {\line(0,1){9}}
\put (52,40) {\line(0,1){9}}
\put (52.1,40) {\line(0,1){9}}
\put (63.2,42) {$t_4$}
\put (62.4,40) {\line(0,1){9}}
\put (62.5,40) {\line(0,1){9}}
\put (62.6,40) {\line(0,1){9}}
\put (70,42) {$t_5$}
\put (81,42) {$t_6$}
\put (17.5,9.5) {$\bullet$}
\put (17.5,39) {$\bullet$}
\put (37.2,39) {$\bullet$}
\put (51.3,39) {$\bullet$}
\put (61.8,39) {$\bullet$}
\put (68.4,39) {$\bullet$}
\put (81.9,39) {$\bullet$}
\end{overpic}
\end{center}
\caption{Notations for normalized flower layouts}
\label{F:f7}
\end{figure}
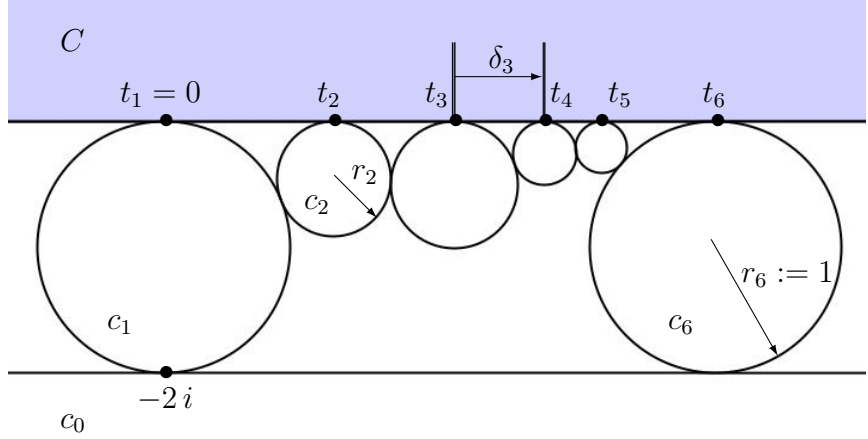

\F{four} provides a sampler of normalized flowers. In \F{four}(a) the
non-contiguous petals $c_3$ and $c_0$ overlap. \F{four}(b) illustrates
an ``extraneous tangency'', as petals $c_{j-1}$ and $c_{j+1}$ are
tangent, even though they are not neighbors in the flower
structure. \F{four}(c) illustrates a flower whose seven petals reach
twice about $C$; necessarily, some of them overlapping one
another. Note that \F{f7} and \F{four}(d) are univalent flowers,
\F{four}(a) is non-univalent, but is un-branched, while \F{four}(c) is
branched.  \F{four}(d) illustrates an extremal situation among
univalent flowers: petals $c_2,\cdots,c_{5}$ all have extraneous
tangencies with $c_0$, yet the petals' interiors are
mutually disjoint.  This illustrates the configuration among
normalized univalent $n$-flowers with the greatest distance between
the end petals $c_1$ and $c_{n-1}$.

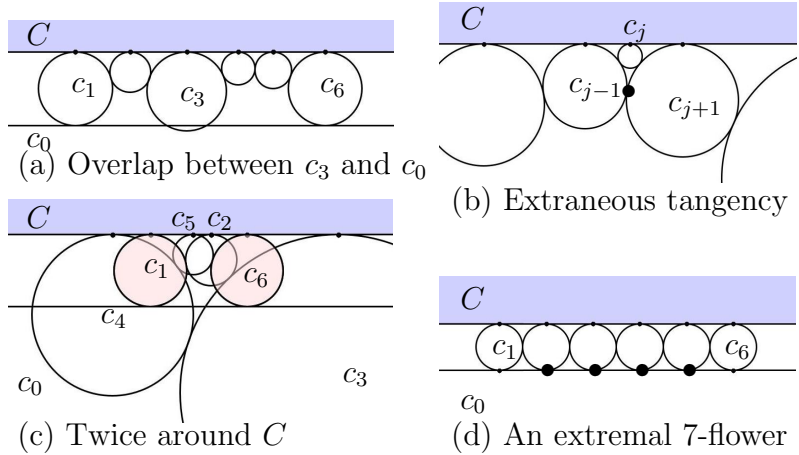
\begin{figure}[h]
\begin{center}
\begin{overpic}[width=.95\textwidth
  ]{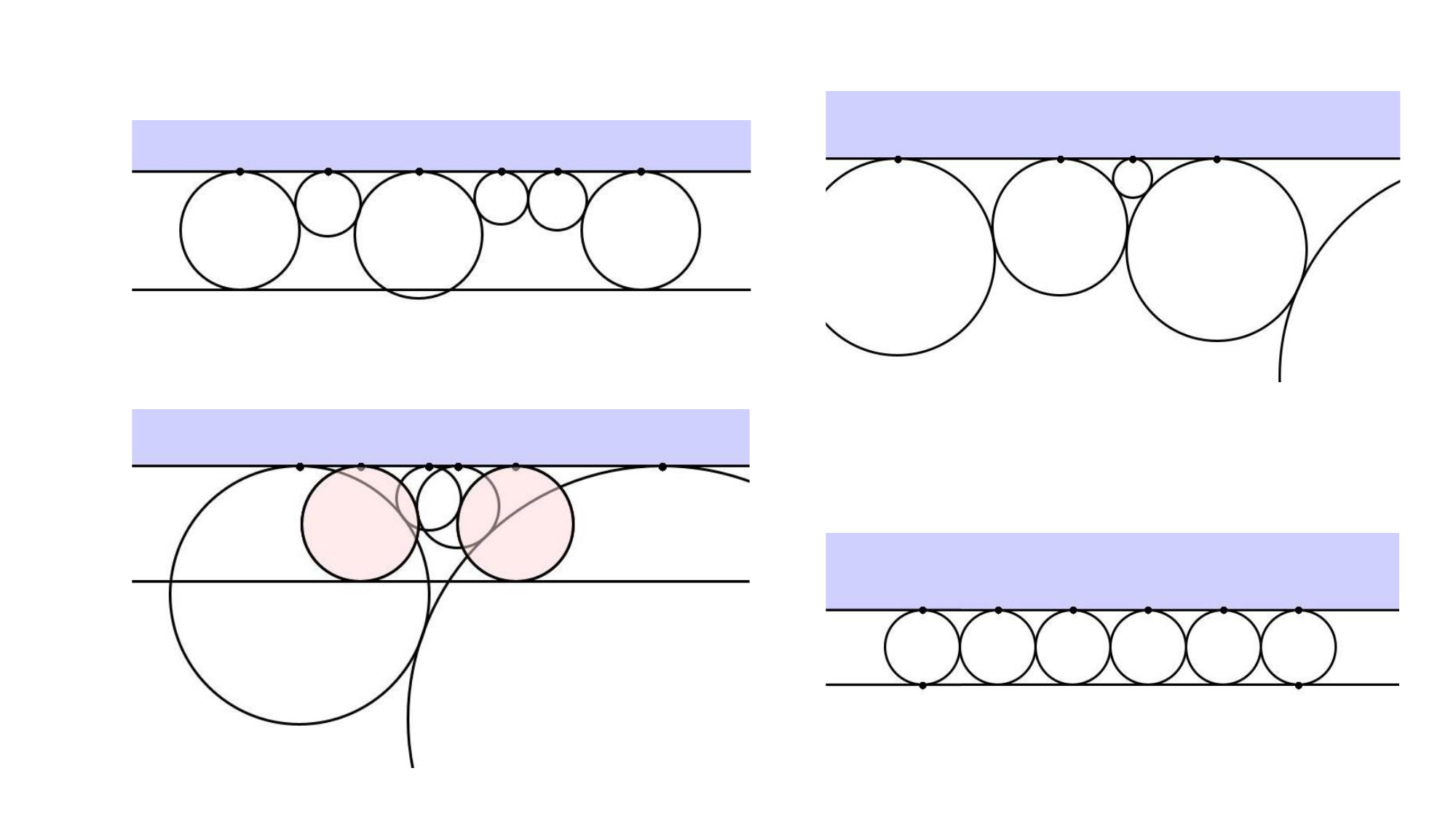}
\put (10,31) {(a) Overlap between $c_3$ and $c_0$}
\put (16,40) {$c_1$}
\put (43.5,40) {$c_6$}
\put (11,34) {$c_0$}
\put (28,39) {$c_3$}
\put (76.75,39.25) {$\bullet$}
\put (58,27) {(b) Extraneous tangency}
\put (77,46.5) {$c_j$}
\put (71,40) {$c_{j-1}$}
\put (82,38) {$c_{j+1}$}
\put (10,1) {(c) Twice around $C$}
\put (10,7) {$c_0$}
\put (24,20) {$c_1$}
\put (31,25.25) {$c_2$}
\put (46,8) {$c_3$}
\put (19.25,14.5) {$c_4$}
\put (27,25.25){$c_5$}
\put (35,19) {$c_6$}
\put (58,1) {(d) An extremal 7-flower}
\put (59,5) {$c_0$}
\put (62.5,11) {$c_1$}
\put (88.25,11) {$c_6$}
\put (67.8,8.4) {$\bullet$}
\put (73.05,8.4) {$\bullet$}
\put (78.25,8.4) {$\bullet$}
\put (83.5,8.4) {$\bullet$}
\put (11,45.25) {$C$}
\put (59,47) {$C$}
\put (11,25) {$C$}
\put (59,16) {$C$}
\end{overpic}
\caption{Examples of normalized flower variety}
\label{F:four}
\end{center}
\end{figure}

\subsection{Flower Layouts} \label{SS:flayout}
Given putative schwarzians $\{s_0,s_1,\cdots,s_{n-1}\}$ for an $n$-flower,
the associated flower can be laid out in normalized form using the
following mechanical process, which relies on computations carried out
in the Appendix.

\vspace{15pt}
\begin{enumerate}
\item[(1)]{\noindent We start with the
  half planes $C$ and $c_0$ and the circle $c_1$ in their normalized
  positions.}

\vspace{10pt}
\item[(2)]{With $c_1$ in place and given the schwarzian $s_1$ for its
  edge, the formulas of (\ref{E:S2}) in the Appendix yield the
  tangency point $t_2$ and radius $r_2$ of $c_2$.  }

\vspace{5pt}
\item[(3)]{With $c_1$ and $c_2$ in place and given $s_2$, we are in
  the ``generic'' Situation~3 of the Appendix, but in the special case
  that $r=1$. In particular, we can place $c_3$ by using (\ref{E:S3})
  to compute radius $r_3$ and displacement $\delta_2$, leading to
  tangency point $t_3$.  }

\vspace{5pt}
\item[(4)]{Hereafter we remain in the generic Situation~3, so we can place
  the remaining petals by inductively applying (\ref{E:S3}) to compute
  radii $r_j$ and the displacements to determine the tangency points $t_j,
  j=4,\cdots,n-1$.  }
\end{enumerate}

\vspace{10pt} At this point we would have the petals of the presumptive
flower all in place. However we can see concretely how the compatibility
conditions mentioned earlier might fail:

\vspace{10pt}
\begin{enumerate}
\item[(a)]{If $r_{n-1}\not=1$, $c_{n-1}$ is fails to be tangent to
  $c_0$ --- the flower does not close up.}

\vspace{5pt}  
\item[(b)]{If $t_{n-1}-t_{n-2}\not=2/(\sqrt{3}(1-s_{n-1}))$ after the
  final application of (\ref{E:S3}) would mean that $s_{n-1}$ is not
the  schwarzian for patch $\{c_{n-1},C\,|\,c_{n-2},c_0\}$.}

\vspace{5pt}  
\item[(c)]{If $t_{n-1}\not=2\sqrt{3}(1-s_0)$ then (\ref{E:S1}) tells us
  that $s_0$ is not the schwarzian for patch $\{c_0,C\,|\,c_{n-1},c_1\}$.}
\end{enumerate}

\vspace{10pt}
\noindent Therefore our work, both theoretical and numerical, depends on a
modification of this process:

\begin{LP} Treating the $n-3$ schwarzians $s_1,\cdots,s_{n-3}$ as
  parameters, we build the
normalized flower as described above up to and including the layout
of $c_{n-2}$. We then {\bf force} closure by setting
$r_{n-1}=1$ and placing the last petal $c_{n-1}$.
\end{LP}

\noindent This Layout Process underlies all the work in this section.
Once the construction has put all petals in place, one can 
{\bf directly compute} the remaining three schwarzians $s_{n-2},
s_{n-1}$, and $s_0$ to fill out the full packing label
$\{s_0,\cdots,s_{n-1}\}$.

\begin{Thm} \label{T:LP}
  Given schwarzians $\{s_1,\cdots,s_{n-3}\}$, the {\bf Layout Process}
  results in a legitimate $n$-flower except in the following two
  situations: (a) when $c_{n-2}$ is tangent to $C$ at infinity or (b)
  when the computed $s_0$ exceeds 1.
\end{Thm}

\begin{proof} The first statement requires no proof, as the mechanics
  are straightforward. As for situation (a), in this case $c_{n-2}$ is
  a half plane, meaning that placement of $c_{n-1}$ simultaneously tangent to
  $C,c_{n-2}$, and $c_0$ is either impossible or ambiguous. Situation
  (b) violates a condition we placed on schwarzians; in this case, 
  petal $c_{n-1}$ ends up to the left of $c_1$. For details on the exceptions,
  go to the closing paragraph of \S\ref{SS:A1}.
\end{proof}
    
\subsection{Important Observations}
In the pencil-and-paper computations leading to the formulas of the
Appendix (and the associated code in {\tt CirclePack}) it became clear
that a new parameter $u=1-s$ is preferrable to the schwarzian $s$
itself. Instead of label $\{s_0,\cdots,s_{n-1}\}$, we will interchangeably use
$\{u_0,\cdots,u_{n-1}\}$, though we continue to treat the $s$-variables as
the proper labels. The author can offer no geometric significance for
this new $u$-variable, but converting the $s$'s in our discussion of
\F{base} to $u$'s may help the reader develop some intuition.  For the
reasons discussed there, we limit our work to $s\in(-\infty,1)$, thus
$u\in(0,\infty)$.
      
Also note that although our normalization picks $c_0$ to be a half
plane, any petal of a flower may be designated as $c_0$ as a simple
matter of indexing. Furthermore, if the order of the petals in a
flower is reversed, the result is still a flower and the order of
schwarzians will have been reversed.  These observations explain,
respectively, the cyclic and the symmetric features in this lemma.

\begin{Lem} \label{L:cr}
  Suppose $\{s_0,\cdots,s_{n-1}\}$ is a packing (edge) label for an
  $n$-flower. If one shifts the order of the schwarzians cyclically or
  reverses the order, the result is again a packing label. This holds
  equally for the $u$-variables $\{u_0,\cdots,u_{n-1}\}$.
\end{Lem}

\vspace{10pt}
\subsection{Un-branched Flowers} Our results are most complete in the case
of un-branched $n$-flowers, where we work step by step starting with
$n=3$.  Examples for degrees $n=3,\,4,\,5$, and $6$ are shown in
\F{degree}. These provide some visual clues to the patterns we will
discuss below. \F{degree}(d) is also cautionary, as it illustrates
three M\"obius equivalent normalized representations of the same
6-flower, differing only by which petal had been designated as $c_0$.

\begin{figure}
\begin{center}
\begin{overpic}[width=1.4\textwidth
  ]{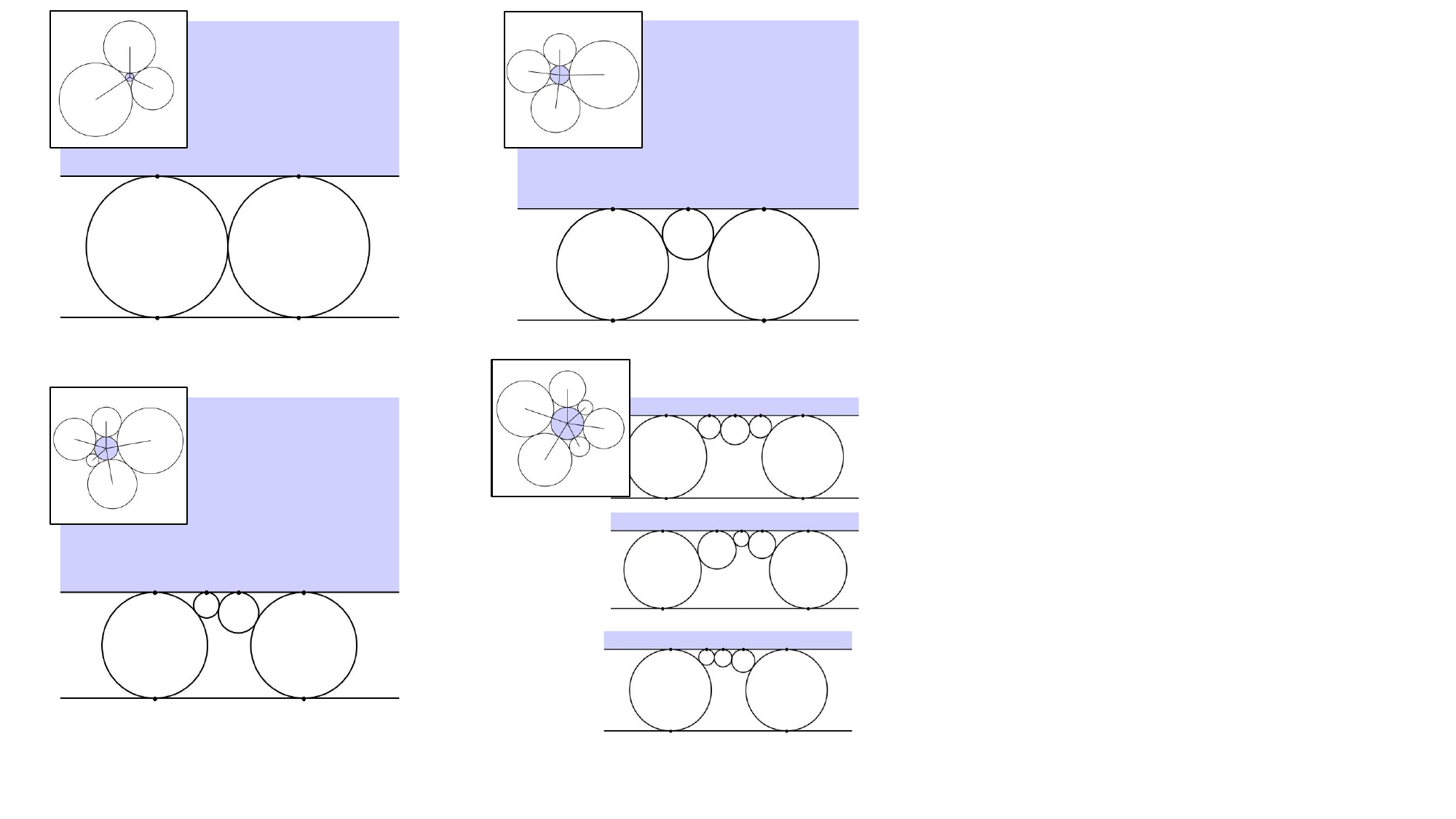}
  \put (10,32) {(a) 3-flower.}
  \put (43,32) {(b) 4-flower.}
  \put (10,3) {(c) 5-flower.}
  \put (43,3) {(d) One 6-flower.}
\end{overpic}
\caption{Normalized flowers}
\label{F:degree}
\end{center}
\end{figure}
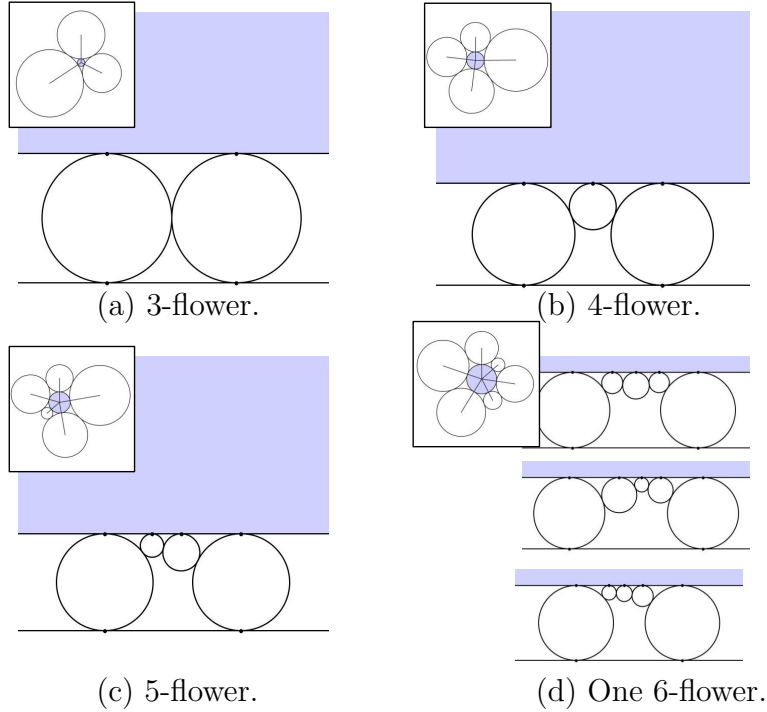

\subsubsection{3-Flowers} \label{SS:3} There is only one 3-flower up
to M\"obius transformations. In particular, the three edge schwarzians
share identical values. Note in \F{base}, that the value 
$s=1-1/\sqrt{3}$ leads to a 3-flower, namely, that formed by
$C_v,C_a,$ and the associated
$\widehat{C}_s$ (enclosing $\infty$).

\begin{Lem} \label{L:3f}
  The intrinsic schwarzian for any edge of a 3-flower is
  $s=1-1/\sqrt{3}$, and hence, $u=1/\sqrt{3}$.
\end{Lem}

This value of $s$ may also occur in higher degree flowers in the case
of extraneous tangency, as seen, for example, in \F{four}(b), where
the schwarzian $s_j$ takes this value.  (It is worth noting that in
circle packings, interior vertices of degree~3 are somewhat extraneous
themselves: the associated circle is determined uniquely by its three
neighboring circles and could be omitted without affecting the
packing's overall geometric structure. On the other hand, omitting
such a circle does affect schwarzians, namely, those of the outer
three edges of its flower.)

\subsubsection{4-Flowers} \label{SS:4}
\F{ffs} shows a sequence of un-branched 4-flowers. Petals $c_1$ and
$c_3$ have radius 1, so it is clear that the size of the shaded petal,
$c_2$, determines the entire normalized 4-flower. There is thus one
degree of freedom. In (\ref{E:S2}) of the Appendix we see the radius
of $c_2$ as a monotone function of $s_1$, so we can use $s_1$ to
parameterize all 4-flowers. (The curious case of a branched $4$-flower
will be displayed in \S\ref{SS:brflow}.)

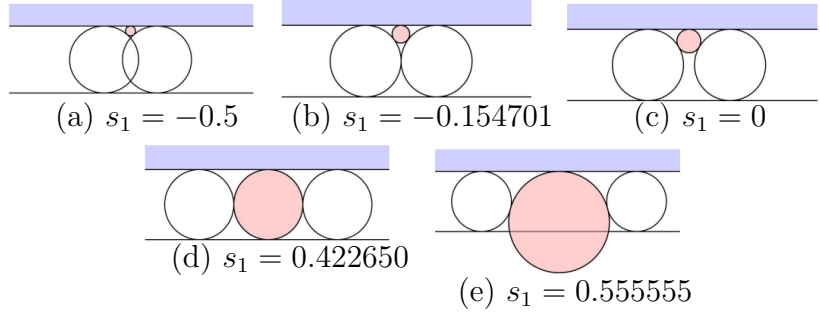
\begin{figure}[h]
\begin{center}
\begin{overpic}[width=.9\textwidth
  ]{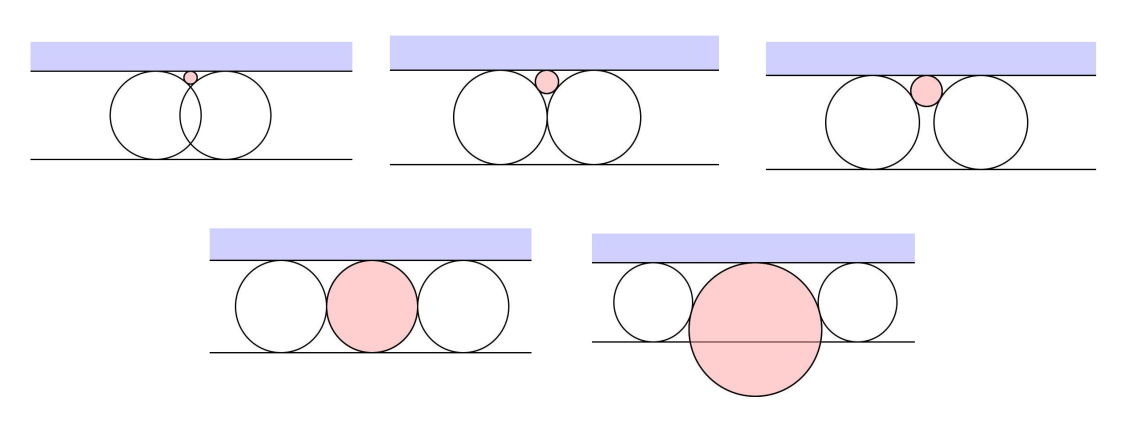}
  \put (8,19.5) {(a) $s_1=-0.5$}
  \put (35.5,19.5) {(b) $s_1= -0.154701$}
  \put (75.5,19.5) {(c) $s_1=0$}
  \put (21.5,3) {(d) $s_1=0.422650$}
  \put (55,-1) {(e) $s_1= 0.555555$}
\end{overpic}
\caption{Variety in 4-flowers}
\label{F:ffs}
\end{center}
\end{figure}

\vspace{5pt}
\begin{enumerate}
\item[$\bullet$] {Apply (\ref{E:S2}) to compute the tangency point and
  radius of $c_2$:
  \begin{equation} \notag
    \delta_1=t_2=2/(\sqrt{3}\,u_1)\ \ \text{ and }
    \ \ r_2=1/(\sqrt{3}\,u_1)^2.
    \end{equation}
}
\item[$\bullet$] {Apply (\ref{E:R3}) with $R=r_2$ and $r=1$ to compute
$u_2=2/(3u_1).$.
}
\end{enumerate}

\vspace{10pt}
\noindent Here we initiate a pattern that we will carry forward
for larger $n$: namely, we define functions 
\begin{equation} \notag
  \fU_4(x)=2/(3x),\qquad \fC_2(x)=x,
\end{equation} 
and note that $u_2=\fU_4(u_1)$ under the constraint that $\fC_2(u_1)>0$.
Furthermore, as we will do for larger degrees, we can engage
Lemma~\ref{L:cr}. With successive left shifts of the parameters we
conclude that $u_3=\fU_4(u_2)$ and $u_0=\fU_4(u_3)$, thereby completing
the full packing label. In particular, note that $u_1=u_3$, $u_0=u_2$,
and $u_1u_2=2/3$. We arrive at a very clean characterization of
packing labels for 4-flowers:

\begin{Lem}\label{L:4f}
  Every un-branched 4-flower has edge schwarzians of the form
  $\{s,s',s,s'\}$ where $s$ and $s'$ satisfy $(1-s)(1-s')=2/3$ (i.e.,
  $uu'=2/3$).  Moreover, the 4-flower is univalent if and only if $s$
  and $s'$ lie in the closed interval
  $I=[1-\frac{2}{\sqrt{3}},1-\frac{1}{\sqrt{3}}]$ (i.e., $u,u' \in
  [\frac{1}{\sqrt{3}},\frac{2}{\sqrt{3}}]$).
\end{Lem}

\F{ffs}(b), (c), and (d) are univalent 4-flowers, with (b) and (d)
representing the extremes of parameter values allowed for
univalence. One should not be mislead by this complete understanding of
4-flowers --- things become increasingly more complicated as the degree
goes up.

\subsubsection{5-Flowers} \label{SS:5}
Our Layout Process tells us that we have 2 degrees of freedom, namely,
$s_1$ and $s_2$. \F{n5} indicates the quantities we can
compute as functions of these for a generic 5-flower.  

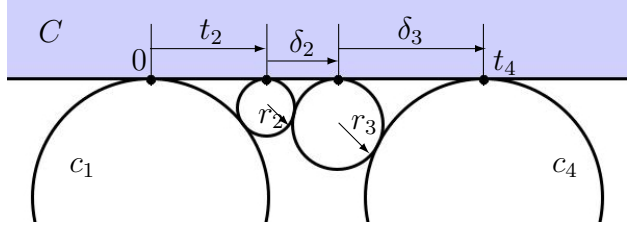
\begin{figure}[h]
\begin{center}
\begin{overpic}[width=.65\textwidth
  ]{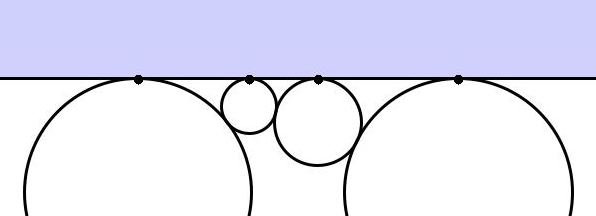}
  \put (5,29) {$C$}
  \put (10,8) {$c_1$}
  \put (88,8) {$c_4$}
  \put (23.2,22) {\line(0,1.0){10}}
  \put (41.9,22) {\line(0,1.0){10}}
  \put (53.5,22) {\line(0,1.0){10}}
  \put (76.9,22) {\line(0,1.0){10}}
  \put (20,24.5) {0}
  \put (78.5,24.5) {$t_4$}
  \put (31,29.5) {$t_2$}
  \put (23,28) {\vector(1.0,0){18.75}}
  \put (45.5,27.5) {$\delta_2$}
  \put (42,26) {\vector(1.0,0){11.5}}
  \put (63,29.5) {$\delta_3$}
  \put (53.5,28) {\vector(1.0,0){23.85}}
  \put (40.5,16) {$r_2$}
  \put (42,19){\vector(2.5,-2.5){3.5}}
  \put (55.5,15) {$r_3$}
  \put (53.5,16){\vector(2.5,-2.5){5}}
\end{overpic}
\caption{Quantities to compute in a normalized 5-flower}
\label{F:n5}
\end{center}
\end{figure}

\vspace{10pt} Quantities resolve cumulatively as we add petals: With
$C,c_0$, and $c_1$ in place, we can apply the computation in the
earlier $4$-degree case to compute $r_2=1/(\sqrt{3}\,u_1)^2$. From
there, successive computations from the Appendix yield various radii
and displacements:

\vspace{5pt}
\begin{enumerate}
\item[$\bullet$]{Applying (\ref{E:S3}) with $u=u_2$, $r=1$, and
  $R=r_2$:
  \begin{equation} \notag
    \delta_2=\dfrac{2}
    {\sqrt{3}\,u_1(3u_1u_2-1)}\ \text{ and }\ 
      r_3=\dfrac{1}{(3u_1u_2-1)^2}.
  \end{equation}
}
\end{enumerate}

\vspace{5pt}
\noindent Here we encounter a constraint: if $(3u_1u_2-1)<=0$, then
$\delta_2$ will be negative or undefined. As we will see later, this
can happen for branched flowers. For the un-branched case we must
impose the condition $(3u_1u_2-1)>0$.

\vspace{5pt}
\begin{enumerate}
\item[$\bullet$] {Using radii $r_3$ from above and the mandated $r_4=1$:
  \begin{equation} \notag
    \delta_3=2\sqrt{r_3}=\dfrac{2}{3u_1u_2-1},
  \end{equation}
}
\item[$\bullet$] {Setting $r=r_2$ and $R=r_3$ in (\ref{E:R3}) implies:
  \begin{equation} \notag
    u_3=\dfrac{u_1+1/\sqrt{3}}{3u_1u_2-1}.
  \end{equation}
}
\end{enumerate}

\vspace{5pt}
\noindent Define the rational function $\fU_5$
and the polynomial $\fC_3$ as follows:
\begin{equation}\notag
  \fU_5(x_1,x_2)=\dfrac{x_1+1/\sqrt{3}}{3x_1x_2-1},\qquad
  \fC_3(x_1,x_2)=3x_1x_2-1.
\end{equation}
Computation of $u_3$ becomes simply $u_3=\fU_5(u_1,u_2)$ under the
constraint $\fC_3(u_1,u_2)>0$. Applying the cyclic property of
Lemma~\ref{L:cr}, we get in succession $u_4=\fU_5(u_2,u_3)$ and
$u_0=\fU_5(u_3,u_4)$. 

\subsubsection{6-Flowers} \label{SS:6}
We work through this one additional case because the full strength of
Situation~3 and (\ref{E:S3}) is first felt with addition of the sixth
petal. (Also because 6-flowers have always occupied a prominent place
in circle packing: in the ``curvature'' language common in this topic,
6-flowers are ``flat''.)

6-flowers involve three degrees of freedom with parameters
$\{s_1,s_2,s_3\}$. We may extend the notations in the previous section
and \F{n5} by adding one additional petal. The computations of $t_2$
and $\delta_2$ and the constraint $(3u_1u_2-1)>0$ are exactly as
earlier. The computation for $\delta_3$, however, needs to
be revisited.

\vspace{5pt}
\begin{enumerate}
\item[$\bullet$]{Applying (\ref{E:S3}) with $r_2$ and $r_3$ as computed
  earlier gives
\begin{align} \notag
  &\delta_3=\delta(u_3,r_2,r_3)=\dfrac{2}
  {\sqrt{3}\,(3u_1u_2-1)(3u_1u_2u_3-u_1-u_3)},\\
  &r_4=\dfrac{\delta_3^2}{4r_3}=\dfrac{1}{3(3u_1u_2u_3-u_1-u_3)^2}.\notag
\end{align}
}
\end{enumerate}

\vspace{5pt}
\noindent Note that $1/\sqrt{r_4}=\sqrt{3}(3u_1u_2u_3-u_1-u_3)$ if
this is positive. Then by completing the flower with petal $c_5$ of
mandated radius 1, we can compute the next label:

\vspace{5pt}
\begin{enumerate}
  \item[$\bullet$]{Applying (\ref{E:R3}) with $r=r_3$ and $R=r_4$ gives
$$ u_4=\dfrac{u_1u_2}{(3u_1u_2u_3-u_1-u_3)}.$$
  }
\end{enumerate}

\vspace{7pt}
\noindent Define the rational function $\fU_6$ and
polynomial $\fC_4$:
\begin{align} \notag
  \fU_6(x_1,x_2,x_3)=&x_1x_2/(3x_1x_2x_3-x_1-x_3),\\
  \fC_4(x_1,x_2,x_3)=&\sqrt{3}(3x_1x_2x_3-x_1-x_3).\notag
\end{align}
Thus $u_4=\fU_6(u_1,u_2,u_3)$ under the assumptions $\fC_3(u_1,u_2)>0$
and $\fC_4(u_1,u_2,u_3)>0$. Cyclic shifts provide the remaining
two labels:
\begin{equation} \notag
  u_5=\fU_6(u_2,u_3,u_4)\ \ \text{ and }\ \ u_0=\fU_6(u_3,u_4,u_5).
\end{equation}

\subsubsection{The General Case} \label{SS:GP} In the $n$-flower
case we are starting with $n-3$ parameters $\{u_1,\cdots,u_{n-3}\}$.
A look at the various formulas from the Appendix suggests focusing
on reciprocal roots of the radii. Here
are the first few expressions. (We introduce a convenient notational
devise that abbreviates a product of $u$'s via multiple subscripts,
allowing us, for example, to write $u_{1,4,5}$ in place of the product
$u_1u_4u_5$.)
\begin{align} \label{E:sqrs}
  &1/\sqrt{r_2}=\fC_2(u_1)=\sqrt{3}u_1,\notag \\
  &1/\sqrt{r_3}=\fC_3(u_1,u_2)=3u_{1,2}-1,\notag \\
  &1/\sqrt{r_4}=\fC_4(u_1,u_2,u_3)={\sqrt{3}}(3u_{1,2,3}-u_1-u_3),  \\
  &1/\sqrt{r_5}=\fC_5(u_1,\cdots,u_4)=9u_{1,2,3,4}-3u_{1,4}-3u_{3,4}-3u_{1,2}+1, \notag \\
  &1/\sqrt{r_6}=\fC_6(u_1,\cdots,u_5)=\notag\\
  &\ \ \ \ \ \ \ \sqrt{3}(9u_{1,2,3,4,5}-3u_{1,4,5}-3u_{3,4,5}-3u_{1,2,5}
  -3u_{1,2,3}+u_1+u_3+u_5).\notag\\
  &\cdots\cdots\notag
\end{align}
For a given $j$, the expression for $1/\sqrt{r_j}$ is
ensured only if $\fC_k(u_1,\cdots,u_{k-1})>0$ for $k=2,\cdots,j$,
and only within $n$-flowers for which $n\ge (j+2)$. Should $\fC_j$ be
negative for some $j$, then the flower will be branched.

With these cautions in mind, the functional notations can become
quite convenient. A label for an $n$-flower may be expressed as an
$n$-vector $p=(u_0,\cdots,u_{n-1})$. We may now write
$\fC_j(u_1,\cdots,u_{j-1})$ as $\fC(p)$, noting that $\fC_j$ uses only
the $j-1$ coordinates $1,2,\cdots,(j-1)$ of $p$. Rewriting
(\ref{E:S3}) in functional notation, we have
\begin{equation}\label{E:ftnl}
  \fC_{j+1}(p)=\sqrt{3}u_j\fC_j(p)-\fC_{j-1}(p),\ 
  3\le j\le n-3.
\end{equation}

When our construction places the last petal, $c_{n-1}$, we compute
$u_{n-2}$ by applying (\ref{E:R3}). In functional notation,
this becomes
\begin{equation}\label{E:ucc}\notag
  u_{n-2}=\fU_n(p)=\dfrac{1+\fC_{n-3}(p)}{\sqrt{3}\,\fC_{n-2}(p)}, n\ge 5,
\end{equation}
where $\fU_n$ depends only on coordinates $1,\cdots,(n-3)$ of $p$.
Here are several of these functions in explicit form:
\begin{align} \label{E:Us}
  &u_2=\fU_4(u_1)=\dfrac{2}{3u_1},\notag \\
  &u_3=\fU_5(u_1,u_2)=\dfrac{u_1+1/\sqrt{3}}{3u_{1,2}-1},\notag \\
  &u_4=\fU_6(u_1,u_2,u_3)=\dfrac{u_{1,2}}
  {3u_{1,2,3}-u_1-u_3}, \\
  &u_5=\fU_7(u_1,u_2,u_3,u_4)=\dfrac{3(3u_{1,2,3}-u_1-u_3)+1/\sqrt{3}}
  {3(3u_{1,2,3,4}-u_{1,2}-u_{1,4}-u_{3,4})+1},\notag\\
  &u_6=\fU_8(u_1,u_2,u_3,u_4,u_5)\notag\\
  &\ \ \ \ \ =\dfrac{3(3u_{1,2,3,4}-u_{1,2}-u_{1,4}-u_{3,4})+2}
  {3(9u_{1,2,3,4,5}-3u_{1,2,3}-3u_{1,2,5}-3u_{1,4,5}-3u_{3,4,5}+u_1+u_3+u_5)}.\notag\\
  &\cdots\cdots\notag
\end{align}

\vspace{10pt} We gather the results for un-branched flowers in this
theorem. Also see the comments that follow.

\begin{Thm} \label{T:unbranched}
  Given $n>3$, the parameters $\{u_1,\cdots,u_{n-3}\}$ are
  part of a packing label for an un-branched $n$-flower if and
  only if
  \begin{equation}\label{E:constraints}
    \fC_j(u_1,\cdots,u_{j-1})>0,\ j=2,\cdots,(n-2).
    \end{equation}
  In this case, these expressions 
  \begin{align}\label{E:final3}
    u_{n-2}=&\fU_n(u_1,\cdots,u_{n-3}),\notag\\
    u_{n-1}=&\fU_n(u_2,\cdots,u_{n-2}),\\
    u_0=&\fU_n(u_3,\cdots,u_{n-1}),\notag
  \end{align}
  allow computation of the three remaining labels.
\end{Thm}

\vspace{5pt}
\noindent This Theorem provides simultaneously a characterization, a
parameterization, and a computational tool for un-branched flowers.
Here are some observations:

\vspace{10pt}
\begin{enumerate}
  \item[(i)] {The functions $\fU_{n}$ and $\fC_{j}$ are particularly
    valuable in light of Lemma~\ref{L:cr}. (Remember, these labels are
    cyclic mod$(n)$.)  In a packing label
    $\{u_0,\cdots,u_{n-1}\}$, any one of its entries $u_j$ may be written
    as $u_j=\fU_n(\sigma)$, where $\sigma$ is the sequence of $n-3$
    entries preceeding $u_j$ (or the reverse of the $n-3$ entries
    following $u_j$).}

    \vspace{8pt}
    \item[(ii)] {Additional relationships pertaining to the normalized
      flower may be extracted from formulas of the Appendix. Here are
      some examples:
    \begin{equation} \label{E:qr}
      \dfrac{1}{\sqrt{r_j}}=\dfrac{\sqrt{3}\,u_{j-1}}{\sqrt{r_{j-1}}}
      -\dfrac{1}{\sqrt{r_{j-2}}},\ \ j=3,4,\cdots,n-1.
    \end{equation}
    \begin{equation} \notag\label{E:rur}
      u_j=\dfrac{\sqrt{\frac{r_j}{r_{j-1}}}+
        \sqrt{\frac{r_j}{r_{j+1}}}}{\sqrt{3}},\ j=2,\cdots,n-3.
    \end{equation}
    \begin{equation} \notag\label{E:2srR}
      \delta_j=2\sqrt{r_jr_{j+1}},\ j=1,\cdots,n-2.
    \end{equation}
    }

  \vspace{8pt}
\item[(iii)] {A careful look at the formulas for a normalized flower
  will show that the radii $r_j$, reciprocal roots $1/\sqrt{r_j}$, and
  tangency points $t_j$ of the petals are all rational functions of
  the $u$-parameters (likewise for the $s$-parameters).
  Moreover, these rational functions have
  their coefficients in the number field $\bQ[\sqrt{3}]$.}
  
  \vspace{8pt}
\item[(iv)] {The rational functions $\fU_n$ and polynomials
  $\fC_j$ also have coefficients in $\bQ[\sqrt{3}]$. Note that
  for each $n$, $\fC_{n-2}(u_1,\cdots,u_{j-3})$ is a pole of
  $\fU_n(u_1,\cdots,u_{n-3})$.  On a practical note, unlike
  expressions such as (\ref{E:sqrs}), the functions $\fU_{\cdot}$ and
  $\fC_{\cdot}$ are independent of the flower normalization.}

  \vspace{8pt}
\item[(v)] {The functions $\fU_{\cdot}$ have intriguing self-referential
  behaviour under cyclic shifts and reversals, and this would seem to make them
  quite special. For example, these expressions show how the $\fU_n$ can
  be nested; here $\pssi{u}{j}{k}$ denotes the sequence
  $\{u_j,\cdots,u_k\}$.
\begin{align} \label{E:mus}
  &u_{n-2}=\fU_n(\pssi{u}{1}{n-3})\notag\\
  &u_{n-1}=\fU_n(\pssi{u}{2}{n-3},\fU_n(\pssi{u}{1}{n-3}))\\
  &u_0=\fU_n(\pssi{u}{3}{n-3},\fU_n(\pssi{u}{1}{n-3}),
  \fU_n(\pssi{u}{2}{n-3},\fU_n(\pssi{u}{1}{n-3}))). \notag
\end{align}
The explicit expressions would be quite messy, but would express the
labels $u_{n-2},u_{n-1},u_0$ of (\ref{E:final3}) directly as functions
of the parameters $u_1,\cdots,u_{n-3}$.  }
\end{enumerate}

\vspace{10pt} We are now in position to describe the parameter space
for un-branched $n$-flowers using vectors $p=(u_0,u_2,\cdots,u_{n-1})$
in $\bR_+^n$. Define $\ccV_n$ as the common solutions of the three
rational expressions of (\ref{E:mus}). In particular, $\ccV_n$ is an
algebraic variety of dimension $n-3$ over the number field
$\bQ[\sqrt{3}]$. There are restrictions, however, as $p$ must reside
in the set $\ccC$ where components $u_j$ are positive and the $\ccC_j$
satisfy (\ref{E:constraints}). The parameter space for un-branched
$n$-flowers is thus $\ccF_n=\ccV_n\cap \ccC\subset \bR^n$.

The parameter space $\ccF_n$ has some rather unique features. Each
point $p\in\ccF_n$ determines a unique flower (that is, unique up to
M\"obius transoformations).  On the other hand, each un-branched
$n$-flower is associated with up to $n$ distinct points $p$, since by
Lemma~\ref{L:cr} one can cyclically permute (the coordinates of)
$p$. We have treated $\{u_1,\cdots,u_{n-3}\}$ as the independent
variables, but in fact, any $n-3$ cyclically successive coordinates
can take on this role. One might wonder about the description of
$\ccC$, defined in terms of inequalities depending on
$u_1,\cdots,u_{n-3}$: most of the individual inequalities in
(\ref{E:constraints}) would fail under cyclic permutation of their
arguments. However, if {\bf all} of the inequalities hold, then
$p\in\ccF_n$ and as a result, each of them individually holds under
cyclic permutation. Likewise, reversing the coordinates of
$p\in\ccF_n$ gives a (generically distinct) point of $\ccF_n$.

\subsection{Univalent flowers}
Among the {\sl un-branched} flowers are the {\sl univalent} flowers,
those whose petals have mutually disjoint interiors. In the study of
discrete analytic functions, univalent flowers are (along with
branched flowers) the most important. Define the subset
$\ccU_n\subset\ccF_n$ to consist of parameters associated with univalent
$n$-flowers.

We develop two collections of inequalities which together
characterize points of $\ccU_n$. The inequalities of this first
collection are very easy to check.
\begin{equation} \tag{A}\label{E:A}
\dfrac{1}{\sqrt{3}}\le u_j\le\dfrac{(n-2)}{\sqrt{3}},\  j=0,\cdots,n-1.
\end{equation}

The second collection (B) of inequalities depends on putting the
flowers in their normalized setting, and checking these takes more work
because a point $p\in \ccF_n$ has $n$ normalized layouts based (as
always) on which petal is designated as ``$c_0$''.  As a result, the
indexing used in laying out a flower --- the indexing occurring in
various formulas --- will generally disagree with the official
indexing of the entries in the given $p$.  We introduce a notational
device to more efficiently state the inequalities.

\vspace{10pt}
{\narrower{\narrower
    \noindent{\bf Notation:} Use $\vec{p}^{\,\,k}$ to indicate a cyclic
    permutation of $p$ which shifts the original coordinate $u_k$ to
    become $u_0$. 

}
}

\vspace{10pt}
\noindent
Note that $p\in \ccU_n$ if and only if $\vec{p}^{\,\,k}\in \ccU_n$ for
all $k=0,\cdots,n-1$. In these inequalities, $r_j(p)$ denotes the radius of
the $j$th petal circle in the normalized layout for $p$.
\begin{equation} \tag{B}\label{E:B}
  r_j(\vec{p}^{\,\,k}) \le 1,\ \forall j=2,\cdots,n-2
  \text{ and }\forall k=0,\cdots,n-1.
\end{equation}

\vspace{5pt}
\begin{Thm} \label{T:univalent}
  Given $n>3$, suppose $p$ represents an un-branched
  $n$-flower, so $p\in\ccF_n$. Then $p$ represents a univalent $n$-flower,
  $p\in\ccU_n$, if and only if $p$
  satisfies the inequalities of (\ref{E:A}) and (\ref{E:B}).
\end{Thm}

\begin{proof}
  \noindent{\sl Necessity:} Suppose $p\in\ccU_n$.  We observed in
  discussing \F{base} that if $s>(1-\frac{1}{\sqrt{3}}$, then
  $\widehat{C}_s$ would overlap $C_a$, contradicting univalence.  This
  gives the lower bound of (A). The upper bound depends on
  $n$. Consider the tangency point $t_{n-1}$ in the normalized
  flower. It is clear that the largest $t_{n-1}$ could be for a
  univalent flower occurs for ``extremal'' flowers like that
  illustrated (for $n=7$) in \F{four}(d). In this case,
  $t_{n-1}=2(n-2)$. Since our flower is univalent, $t_{n-1}\le
  2(n-2)$. On the other hand, by (\ref{E:S1}),
  $2\sqrt{3}u_0=t_{n-1}$. We conclude that
  $u_0\le\frac{(n-2)}{\sqrt{3}}$.  Lemma~\ref{L:cr} tells us that any
  $u_j$ can be treated as $u_0$, so $u_j\le\frac{n-2}{\sqrt{3}}$. 

  We prove (\ref{E:B}) by constradiction: suppose (\ref{E:B}) fails
  for some $k$ and $3\le j\le n-1$. Then in the
  shifted indexing of $\vec{p}^{\,\,k}$, the petal circle $c_j$, having
  radius greater than 1 would necessarily overlap the half plane
  $c_0$, contradicting univalence. We
  have established the necessity of (\ref{E:A}) and (\ref{E:B}).

  \vspace{10pt}
  \noindent{\sl Sufficiency:} Suppose the point $p\in\fF_n$ satisfies
  (\ref{E:A}) and (\ref{E:B}) but its flower is not univalent.  That
  is, suppose there exists some pair $c_j,c_k$ of petal circles that
  overlap, $0\le j<k\le n-1$.  Suppose first that there is a single
  petal between $c_j$ and $c_k$, so $k=j+2$. In the normalized layout
  for the shifted point $\vec{p}^{\,\,j}$, the two end circles of the
  layout would overlap, implying by (\ref{E:S1}) that
  $2\sqrt{3}\,u_j<2$ and hence $u_j<1/\sqrt{3}$, violating
  (\ref{E:A}).  On the other hand, if $c_j$ and $c_k$ are separated by
  at least two petals, then with the new indexing of
  $\vec{p}^{\,\,j}$, the petal $c_j$ becomes $c_0$, and $c_k$ becomes
  a petal strictly between $c_1$ and $c_{n-1}$. Since that petal
  overlaps $c_0$, its radius must exceed 1, contradicting
  (\ref{E:B}). This completes the proof of sufficiency.
\end{proof}

In practice, (\ref{E:A}) is trivial to check, while (\ref{E:B}) takes
most of the work. Although the normalized flowers for shifts
$\vec{p}^{\,\,k}$ are all M\"obius images of one another, one must
still check (\ref{E:B}) for each of the normalizations in turn. In a
given normalization, one can use the quasi-recursive expression
(\ref{E:qr}) along with the fact that $r_0=\infty$ and $r_1=1$ to
quickly see if an instance of (\ref{E:B}) is violated.

\subsection{Branched Flowers} \label{SS:brflow}
Now we address the more difficult setting of branched flowers, which
are essential in the study of discrete analytic functions. \F{br_ex_1}
and \F{br_ex_2} provide examples.

\begin{figure}[h]
\begin{center}
\begin{overpic}[width=.98\textwidth
  ]{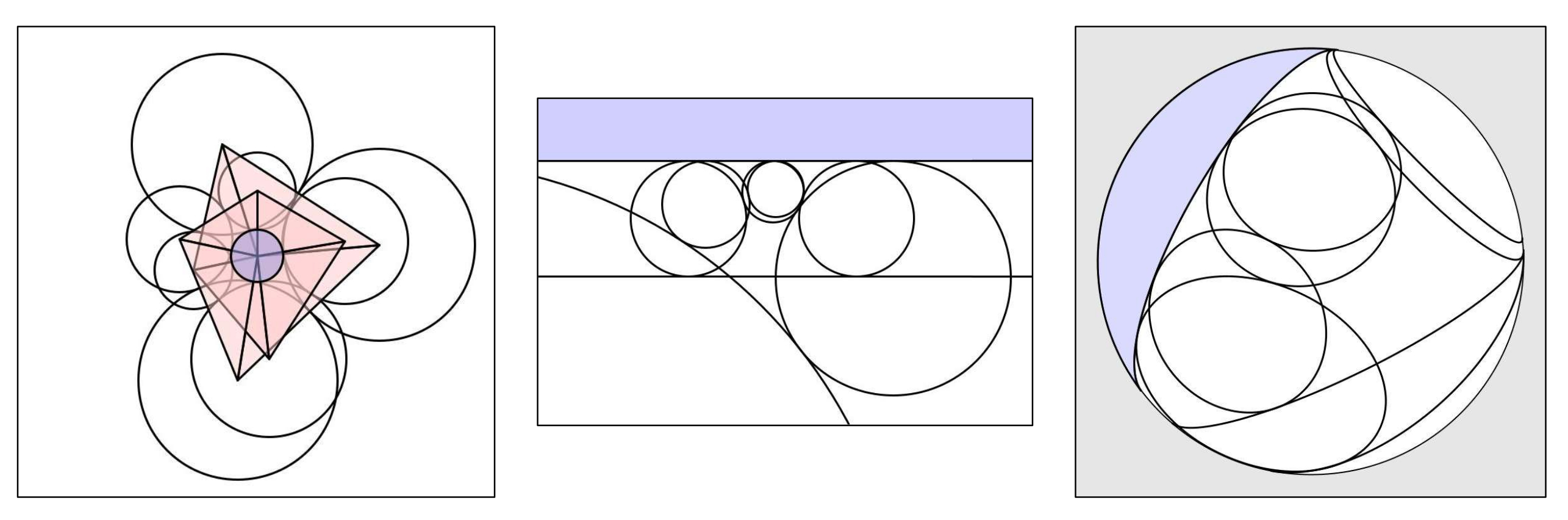}
\end{overpic}
\end{center}
\caption{Images of a simply branched 8-flower}
\label{F:br_ex_1}
\end{figure}

\F{br_ex_1} has three images of the same branched 8-flower
under different M\"obius transformations: one euclidean, one
normalized, and one on $\bP$. Flowers laid out using schwarzians
can easily end up with petals enclosing $\infty$, as on the right;
visualization on the sphere is a slippery business.

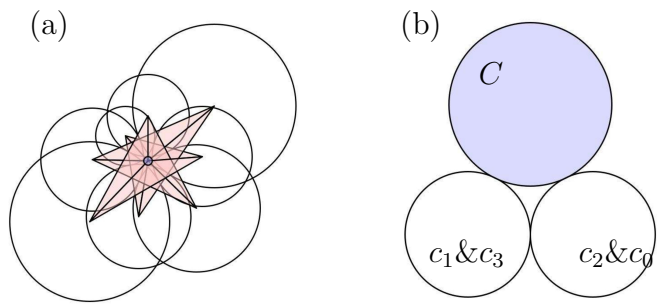
\begin{figure}[h]
\begin{center}
\begin{overpic}[width=.75\textwidth
  ]{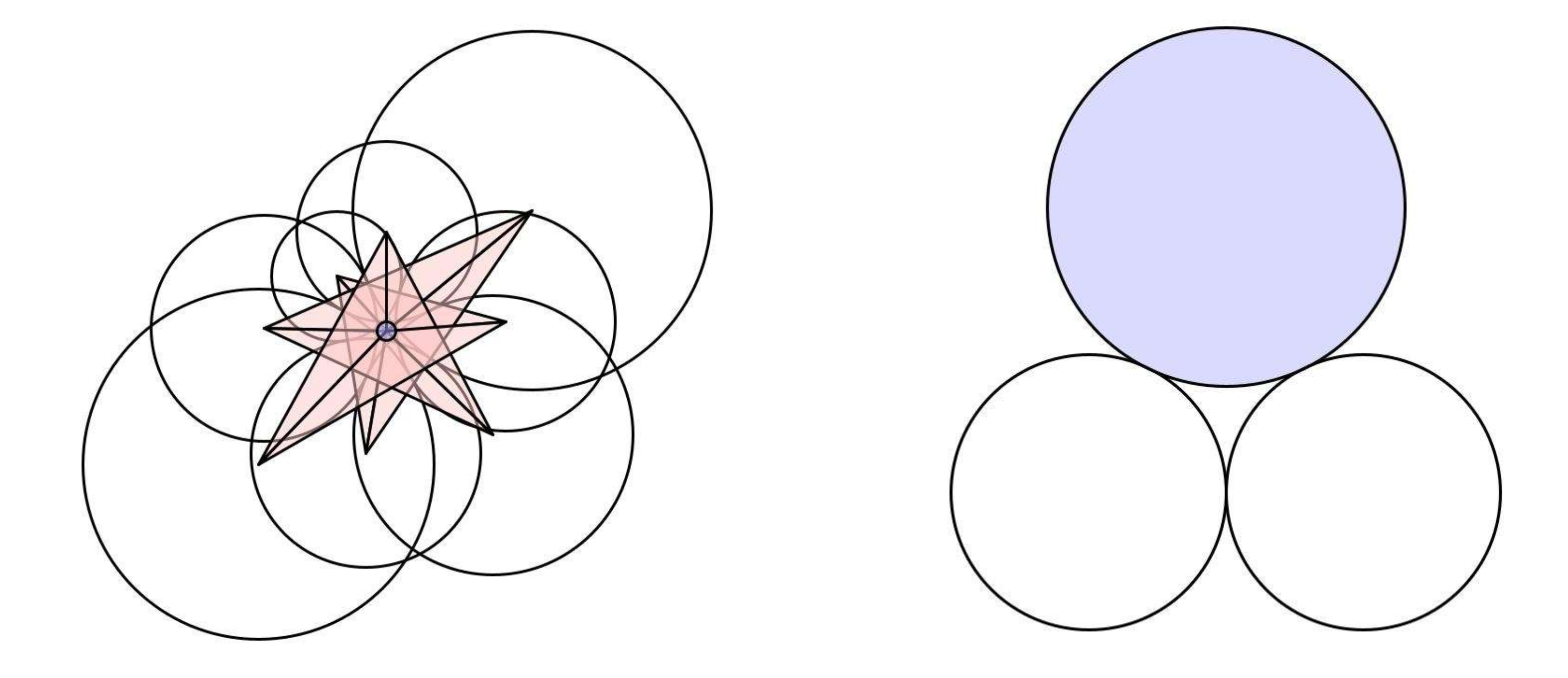}
  \put (8,40) {(a)}
  \put (60,40) {(b)}
  \put (64,8) {$c_1\&c_3$}
  \put (85,8) {$c_2\&c_0$}
  \put (71,33) {$C$}
\end{overpic}
\end{center}
\caption{Two branched flowers}
\label{F:br_ex_2}
\end{figure}

\F{br_ex_2}(a) illustrates a typical flower with branching of order two,
its 8 petals (and 8 faces) wrapping three times around $C$.  \F{br_ex_2}(b), on
the other hand, is a cautionary example: it displays a branched
4-flower. What constitutes a flower depends on context: Are extraneous
tangencies allowed? Must the flower be able to live in a larger circle
packing? Etc. The standard requirement is that a flower wrapping $k$
times about its center will require $n\ge(2k+1)$ petals. This would
hold if we required schwarzians $s$ strictly less than 1 (i.e.,
$u>0$). In light of Lemma~\ref{L:4f}, the 4-flower of \F{br_ex_2}(b) is a
limit: let $u\downarrow 0$ so $u'=2/(3u)\uparrow \infty$. The petals
$c_1$ and $c_3$ are identical, as are $c_2$ and $c_0$. This is not a
situation that would occur in the practice of discrete function
theory.

\vspace{10pt}
The adjustments in our machinery to accommodate branching are
contained in (\ref{E:S4}). Given $n-3$ parameters, we are
still able to compute the remaining three to form a packing label.
Indeed, we can still write $u_{n-2}=\fU_n(u_1,\cdots,u_{n-3})$, but we
must accept the function $\fU_n(\cdots)$ as representing an {\bf algorithm}
rather than an explicit formula.

\vspace{10pt} How does this fit in the broader enterprise of finding
packing edge labels $S$ for more general complexes $K$? In
\S\ref{SS:dculty}, ``criteria'' and ``monotonicities'' were identified
as valuable ingredients. To check an edge label $S$ for $K$, visit
each interior vertex and confirm the expressions for its schwarzians
involving the appropriate $\fU_n$. This is cumbersome even for an un-branched
$n$-flower, as one has to check several constraint functions $\fC_j$ and
apply $\fU_n$ three times. Moreover, if the flower is branched, $\fU_n$
has no explicit formula but rather requires computations involving
normalization. Despite the computational challenge, we do have packing
criteria so that we can check labels $S$ on $K$.

The outlook for monotonicities in $S$ remains quite cloudy.  An
adjustment of the label for one edge typically affects two
flowers. Judging whether this adjustment is favorable in moving us
toward a packing label is very challenging. This is particularly the
case when branching is involved; without explicit expressions, we
can't use classical tools such as differentials. This is unfortunate
as branched packings are among our key targets.

\vspace{10pt} In conclusion, though branched packing are handled quite
routinely in the disc or the plane, where the controlling parameters
are radii and angle sums, the sphere is a tougher environment. Whether
there is some way to deploy the properties we have developed for
flowers to the computational hurdles with larger packings remains an
open question.

\section{Special Classes of Flowers} \label{S:sclass}
In this final section, we consider five distinguished families of
flowers whose schwarzians can be computed explicitly. Examples are
shown in their native environments in \F{specials}. Once again, the
author would point to the intriguing relations that emerge in these
beautiful but elementary cases.

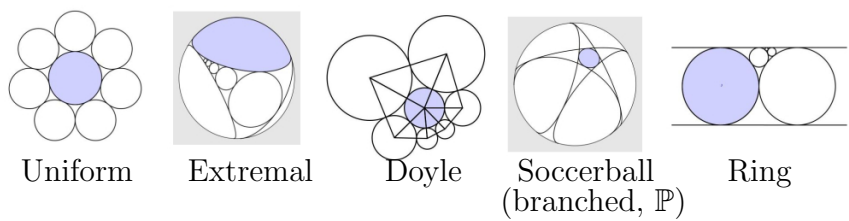
\begin{figure}[h]
\begin{center}
\begin{overpic}[width=.9\textwidth
  ]{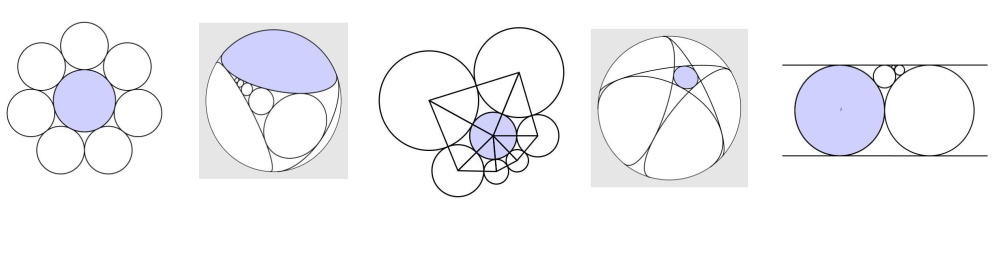}
  \put (2,4.5) {Uniform}
  \put (21.5,4.5) {Extremal}
  \put (60,4.5) {Soccerball}
  \put (58,1) {(branched, $\bP$)}
  \put (44.5,4.5) {Doyle}
  \put (84.5,4.5) {Ring}
\end{overpic}
\caption{Examples of special flower classes}
\label{F:specials}
\end{center}
\end{figure}

\vspace{10pt}
\subsection{Uniform Flowers} \label{SS:uniform}
There is a unique ``uniform'' univalent $n$-flower for each $n$.  The
model $n$-flower might be a euclidean one whose $n$ petals all share
the same radius, as in \F{specials}. These are a natural, unbiased
starting point when first encountering flowers, and deviations from
uniformity can be useful (even in computations; see the ``uniform
neighbor model'' of \cite{CS03}).  \F{uniform} illustrates several
uniform flowers in our normalized form. Symmetry insures that all $n$
schwarzians are identical and one can observe reflective symmetry in
the layouts.

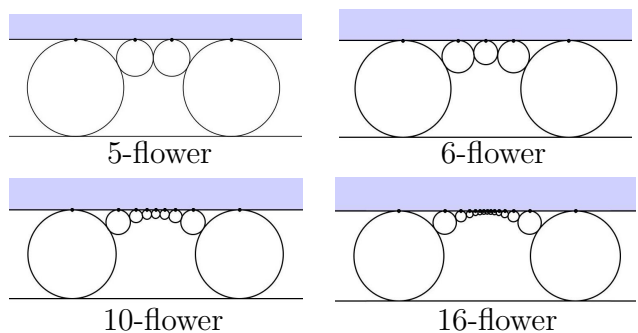
\begin{figure}[h]
\begin{center}
\begin{overpic}[width=.9\textwidth
  ]{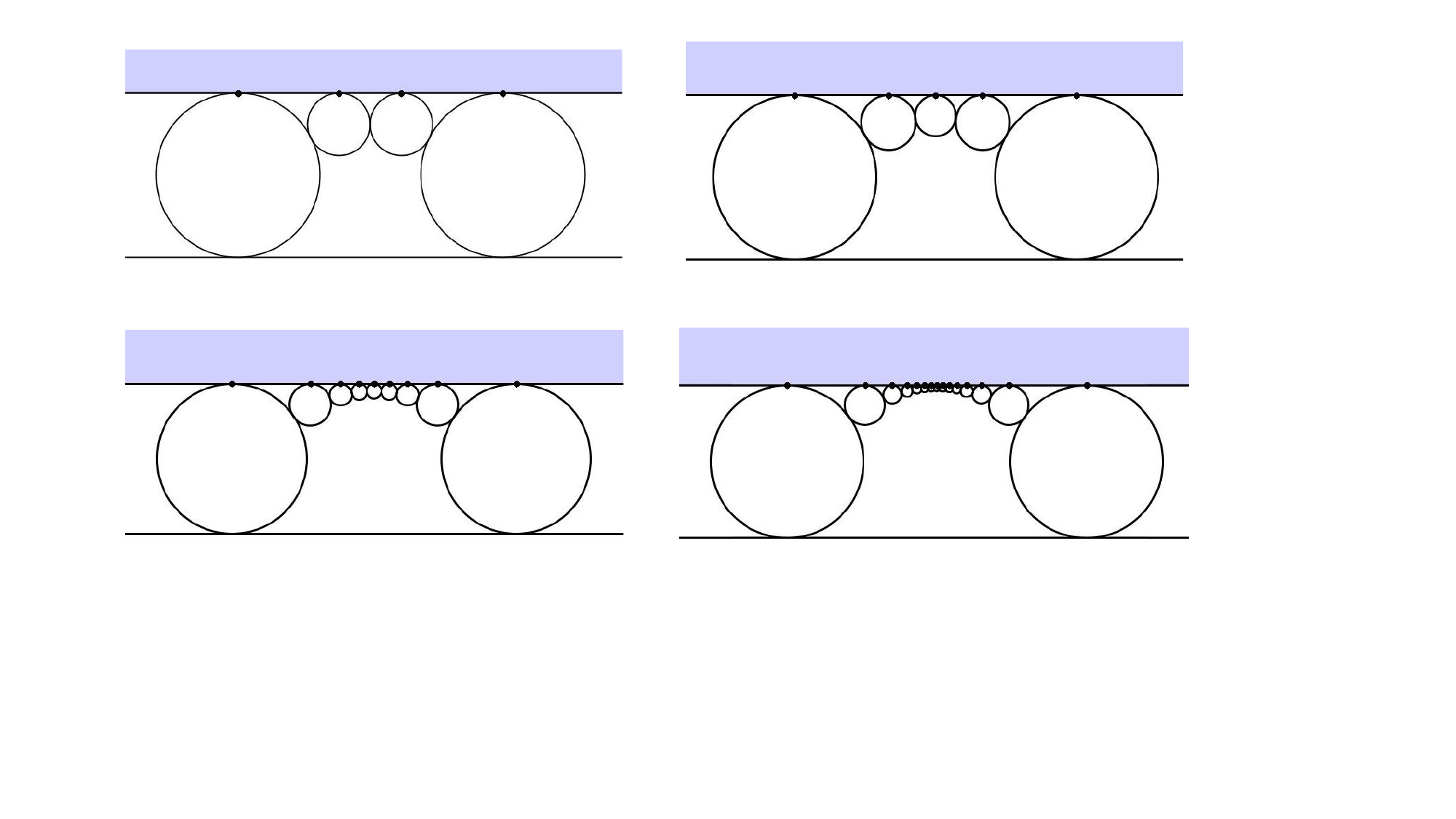}
  \put (20,35.5) {5-flower}
  \put (59,35.5) {6-flower}
  \put (19.5,16) {10-flower}
  \put (58.5,16) {16-flower}
\end{overpic}
\caption{Samples of ``uniform'' flowers}
\label{F:uniform}
\end{center}
\end{figure}

\vspace{10pt} There is some regularity that your eye may pick up on
here. The next figure explains that feeling: in every case all $n$
petals are tangent to a common circle, shaded here. In the model
euclidean setting this circle is the one circumscribing the
flower. (Note, conversely, if such a circle tangent to all the petals
exists, then the flower is uniform.)

\begin{figure}[h]
\begin{center}
\begin{overpic}[width=.5\textwidth
  ]{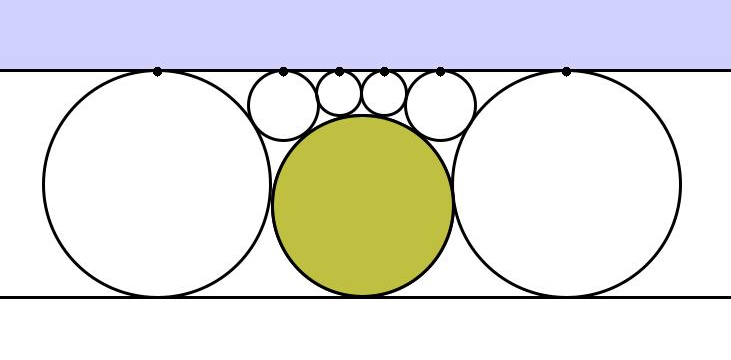}
\end{overpic}
\caption{The hidden disc}
\label{F:unif_disc}
\end{center}
\end{figure}

\vspace{10pt} The key question, of course: ``What is the schwarzian
for a uniform $n$-flower?'' We will label this value as $\fs_n$. Three
cases are already known: $\fs_3=1-1/\sqrt{3}$, $\fs_4=1-\sqrt{2/3}$
(from Lemma~\ref{L:4f}), and $\fs_6=0$.  In \S\ref{SS:A2} of the
Appendix we establish the closed form
$\fs_n=1-\frac{2}{\sqrt{3}}\cos(\pi/n)$. Here is a sampling of values:
\begin{align} \notag
  \fs_3=&\ 1-1/\sqrt{3}\sim 0.422650\notag &{\ }
  \fs_4=&\ 1-\sqrt{2/3}\sim 0.183503\notag \\
  \fs_5\sim&\ 0.065828\notag &{\ }
  \fs_6=&\ 0\notag \\
  \fs_9\sim&\ -0.085064\notag &{\ }
  \fs_{12}\sim&\ -0.115355\notag \\
  \fs_{20}\sim&\ -0.140485\notag &{\ }
  \fs_{50}\sim&\ -0.152422\notag 
\end{align}
The same computations work for branched uniform flowers, giving the
more general formula in \S\ref{SS:A2}(\ref{E:alpha}), which will be
used in \S\ref{SS:soccer}.

\vspace{10pt}
\subsection{Extremal flowers} \label{SS:extremal}
There is a unique ``extremal'' univalent $n$-flower for each
$n$. Indeed, suppose one of the schwarzians $s$ of a univalent
$n$-flower equals the lower bound $1-(n-2)/\sqrt{3}$. Designating that
as the petal $c_0$, (\ref{E:S1}) implies that the normalization
process can only lead to a flower like that of \F{extremal}(a) (for
$n=7$).

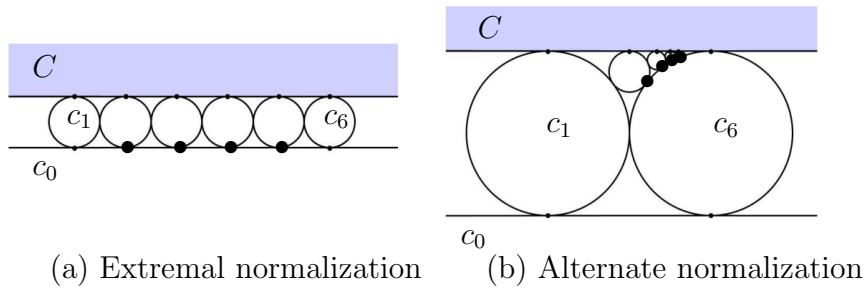
\begin{figure}[h]
\begin{center}
\begin{overpic}[width=.9\textwidth
  ]{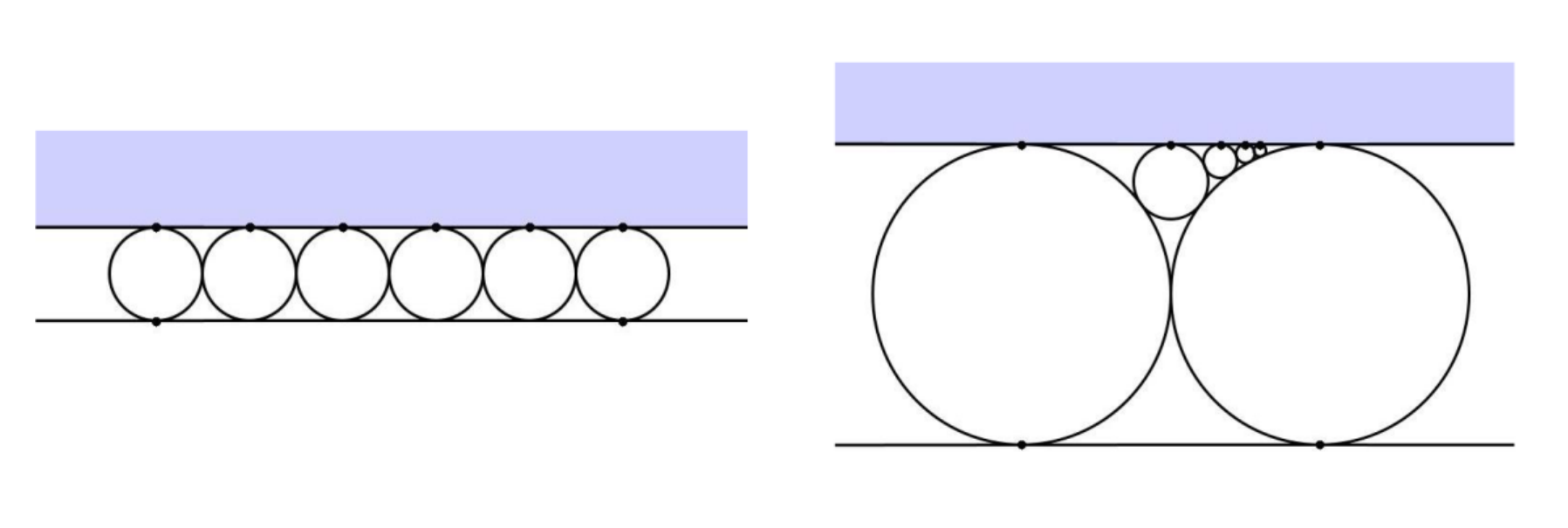}
  \put (5,21.5) {$C$}
  \put (57,26) {$C$}
  \put (9,16) {$c_1$}
  \put (39,16) {$c_6$}
  \put (5,10) {$c_0$}
  \put (55,2) {$c_0$}
  \put (65,15) {$c_1$}
  \put (84.5,15) {$c_6$}
  \put (15.2,12.35) {$\bullet$}
  \put (21.35,12.35) {$\bullet$}
  \put (27.25,12.35) {$\bullet$}
  \put (33.2,12.35) {$\bullet$}
  \put (75.85,20.1) {$\bullet$}
  \put (77.6,21.75) {$\bullet$}
  \put (78.75,22.5) {$\bullet$}
  \put (79.75,22.9) {$\bullet$}
  \put (7,-2) {(a) Extremal normalization}
  \put (58,-2) {(b) Alternate normalization}
\end{overpic}
\caption{Normalizations of an ``extremal'' univalent flower}
\label{F:extremal}
\end{center}
\end{figure}

\vspace{10pt} As can be seen in this normalization, extraneous
tangencies allow $c_1$ and $c_{n-1}$ to serve as centers of 3-flowers,
implying that their schwarzians $s_1,s_{n-1}$ both equal
$\fs_3=1-1/\sqrt{3}$.  Next, consider any of the remaining petals
$c_j, j=2,\cdots,n-2$.  First, $c_j$ acts as the center of a 3-flower
if we include just one of its horizontal neighbors $c_{j+1}$ or
$c_{j-1}$ along with the two half planes. As part of this 3-flower,
the horizontal edge to the neighbor has schwarzian
$1-1/\sqrt{3}$. Note at the same time, that $c_j$ acts as the center
of a 4-flower if you throw in both its horizontal neighbors. Knowing
the horizontal schwarzians, we can apply Lemma~\ref{L:4f} to compute
the schwarzian $1-2/\sqrt{3}$ for the vertical edges. In particular,
we have the following
\begin{Lem}
  For every $n\ge 3$ there is a unique univalent extremal $n$-flower,
  and its set of schwarzians is given by
  \begin{equation} \notag
    \{\frac{\sqrt{3}-1}{\sqrt{3}},\ \ \frac{\sqrt{3}-2}{\sqrt{3}},
    \ \ \frac{\sqrt{3}-2}{\sqrt{3}}, 
    \ \cdots,\ \frac{\sqrt{3}-2}{\sqrt{3}},
        \ \ \frac{\sqrt{3}-1}{\sqrt{3}},\ \ \frac{(n-2)}{\sqrt{3}}\}.
  \end{equation}
\end{Lem}
\noindent \F{extremal}(b) shows a alternate normalization of
the same extremal 7-flower, hence with a shifted list of the same
schwarzians.

\vspace{10pt}
\subsection{Soccerball Flowers} \label{SS:soccer}
We go into detail about the soccerball circle packings discussed in
\S\ref{SS:bdaf} and displayed there in \F{merom}(a). The highly
symmetric nature of these packings allows us to calculate their
intrinsic schwarzians explicitly --- a rare opportunity.

The construction of the soccerball packings are described more
fully in \cite{kS05}. Briefly, the
complex $K$ has 42 vertices, 12 of degree~5 and the rest of
degree~6. The packing $P_K$ on the left in \F{merom}(a) is the
maximal packing for $K$ and is called the soccerball packing because
its dual faces form the traditional soccer ball pattern, breaking the
sphere into 5 and 6-sided polygonal regions.  That on the right, $P$,
is a branched packing for $K$, with simple branching at each of the
degree~5 vertices. The mapping $F:P_K\rightarrow P$ is a key instance
of a discrete rational function.

The ubiquitous symmetries within $K$, $P_K$, and $P$ allow one to
establish these facts: (a) $K$ has only two types of edges, those with
ends of degrees~5 and 6 and those whose ends are both degree~6.  (b)
In each of $P_K$ and $P$, all circles of degree~5 share one radius,
while all of degree~6 share another. (c) These facts imply that in
each of $P_K$ and $P$, there are only two intrinsic schwarzians: $s$
for edges of degree~5 flowers and $s'$ between degree~6 circles.  (d)
And finally, it follows that the degree~5 flowers are uniform while
the degree-6 schwarzians take the alternating pattern
$\{s,s',s,s',s,s'\}$

Working in $P_K$ first, (d) above and (\ref{E:S4}) implies
$u=1-s=\frac{2}{\sqrt{3}}\cos(\pi/5)$. It remains to compute
$u'=1-s'$. However, it turns out we can work more generally.  Consider
any univalent degree~6 flower whose label has the alternating form
$\{u,u',u,u',u,u'\}$. Examples are shown in \F{spec6}.

\begin{figure}[h]
\begin{center}
\begin{overpic}[width=.8\textwidth
  ]{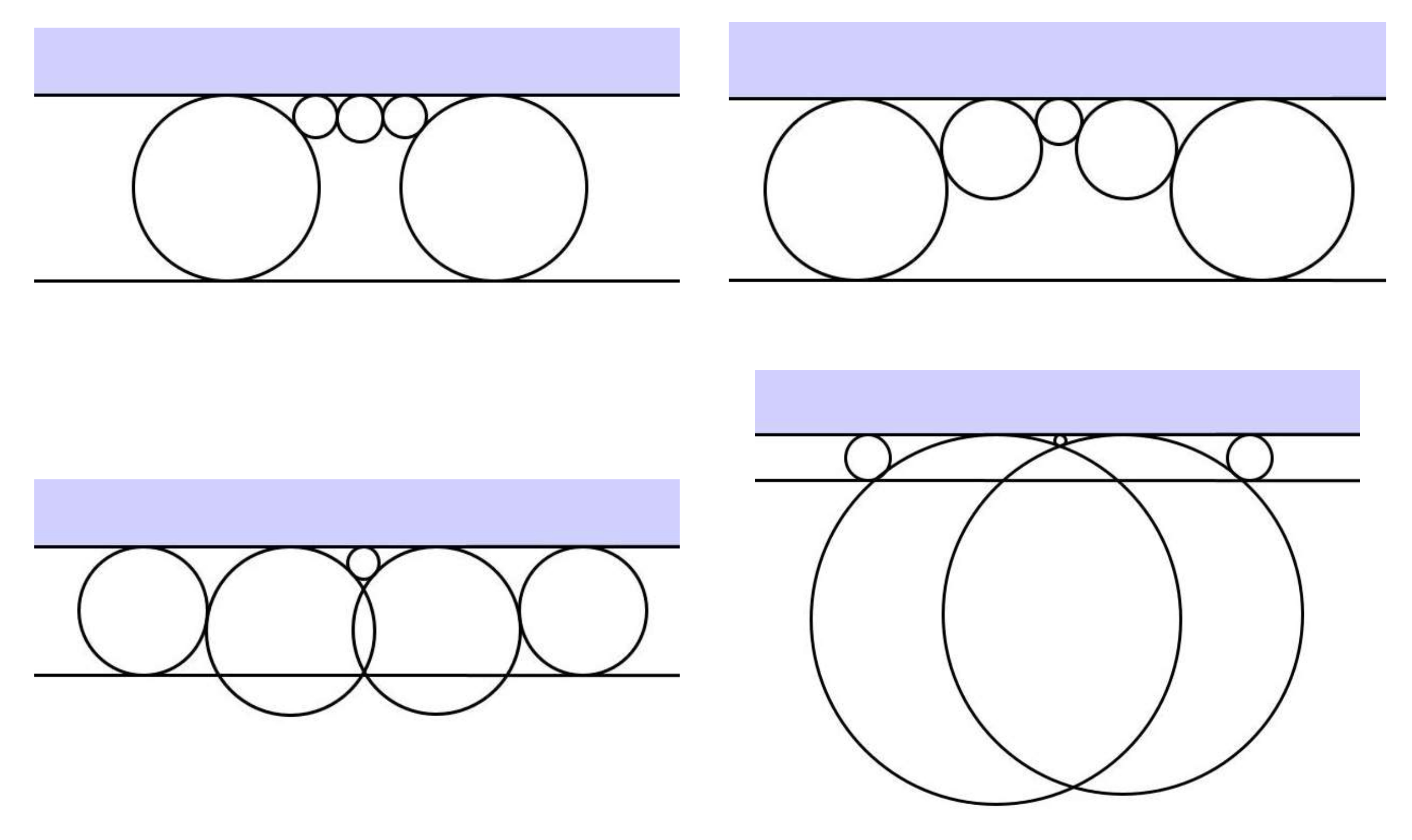}
  \put (15,35) {(a) $u=1.20650...$}
  \put (65,35) {(b) $u=0.77700...$}
  \put (15,4) {(c) $u=0.66666...$}
  \put (65,-2) {(d) $u=0.20000...$}
\end{overpic}
\caption{6-flowers with alternating schwarzians}
\label{F:spec6}
\end{center}
\end{figure}
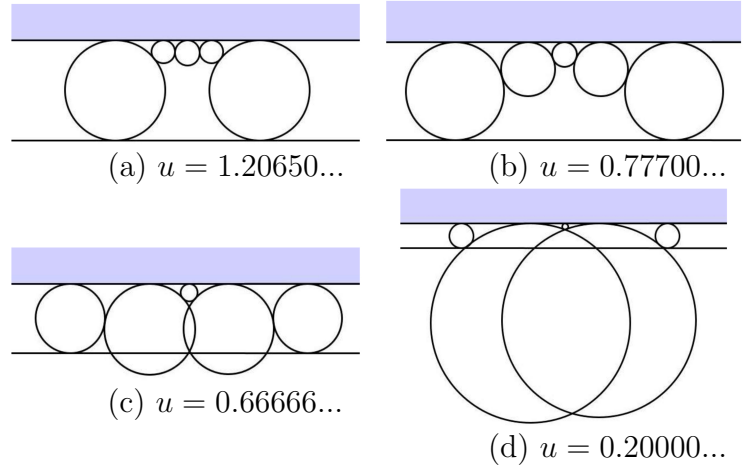

Applying $\fU_6$ and some simplification, we have
\begin{equation} \label{E:uup}
  u=\fU_6(u',u,u')\Longrightarrow
  u=\dfrac{u'u}{3u'uu'-u'-u'}\Longrightarrow
  uu'=1.
\end{equation}
Surprise: $uu'=1$. In our particular case, we conclude for $P_K$ that
\begin{align} \label{E:PKss}
  &s=1-\frac{2}{\sqrt{3}}\cos(\pi/5)\sim 0.065828 \\
  &s'=1-\frac{\sqrt{3}}{2}\sec(\pi/5)\sim -0.070466. \notag
\end{align}

Another feature of these special degree~6 flowers might catch your eye
in \F{spec6}(a), (b), and (c): one can show that for any $s\le\fs_3$,
a normalized flower with label $\{s,s',s,s',s,s'\}$ will lead to
circle $c_4$ having radius $r_4=1/4$. The flower in \F{spec6}(d) shows
that this fails for non-univalent cases: when $s>\fs_3$, $c_3$ and
$c_5$ overlap.  (Incidently, when $s$ is precisely equal to $\fs_3$,
$c_3$ and $c_5$ are tangent and the value 1/4 for $c_4$ comes directly
from the Descartes Circle Theorem.)

\vspace{10pt} We turn now to the branched packing $P$, on the right
in \F{merom}(a). Each of
the 12 degree flowers is branched, so the Riemann-Hurwitz relations
imply that $P$ is a 7-sheeted covering of the sphere. The packing is
very difficult to interpret: the degree~5 circles are smaller now, but
the degree~6 are quite huge --- nearly hemispheres. This and the seven
sheetedness make individual degree~6 circles very hard to distinquish
in \F{merom}(a), so an isolated 5-flower is shown in \F{specials}. Many
facts about $P_K$ persist in $P$: as before there are just two
schwarzians, $s,s'$; the degree~5 vertices have uniform flowers; and
$s,s'$ alternate in the degree~6 flowers. The normalized layout for
one of these 6-flowers is shown in \F{spec6}(d).

We can now compute $u=1-s$ using the general expression
(\ref{E:alpha}) of the Appendix. A simple branched 5-flower wraps
twice around its central circle; this means for a uniform flower that
each face subtends an angle $\theta=4\pi/5$, implying that
$\alpha=\theta/2=2\pi/5$.  Using (\ref{E:alpha}) and then
(\ref{E:uup}) we conclude for $P$ in \F{merom}(a),
\begin{align} \label{E:Pss}
  &s=1-\frac{2}{\sqrt{3}}\cos(2\pi/5)\sim 0.643178  \\
  &s'=1-\frac{\sqrt{3}}{2}\sec(2\pi/5)\sim -1.802517.\notag
\end{align}

As a final comment, we observe that for this very special complex $K$,
our analysis extends to other pairs of schwarzians $s,s'$ satisfying
(\ref{E:uup}), meaning such that $(1-s)(1-s')=1$. One obtains a family
of projective circle packings $P_s$ which live on coned spheres. I am
not prepared to address the range of possible values --- an
interesting question in itself --- but one can choose $s$ to
interpolate between (\ref{E:PKss}) and (\ref{E:Pss}) and certainly to
extend beyond that range. For each appropriate $s$, the process
described in \S\ref{S:layout}, starting with an arbitrary mutually
tangent triple of circles for some base face and then using the
schwarzians to lay out the remaining circles, generates a circle
packing $P_s$. Only when $s$ takes its value from (\ref{E:PKss}) or
(\ref{E:Pss}) (i.e., $P_s=P_K$ or $P_s=P$, respectively) is $P_s$ a
traditional circle packing on the Riemann sphere $\bP$. In all other
cases, the face-by-face construction produces a topological sphere
$\bP_s$ with spherical geometry, save for the 12 points associated
with degree~5 vertices. These are clearly 12 cone points. The symmetry
group of $K$ is the dodecahedral group, so there exists a M\"obius
transformation $\cP_s$ making the singularities indistinguishable,
that is, they all share a common cone angle $\gamma$. To illustrate,
if $s=-0.321284$, then {\tt CirclePack} tells us that
$\gamma=3\pi$. It is left to the curious reader to work out the
precise relationship between $s$ and $\gamma$.

Incidently, imposing symmetry was necessary here since cone angles are
subject to change under M\"obius transformations. Irrespective
of the construction of $P_s$ (i.e., of the initial face $f_0$), the
traditional angle sums $\theta_v$ at all vertices of degree~6 will be
$2\pi$. However, the angle sums at the degree~5 vertices may no longer
share a common value. The only exceptions are our two special cases:
when $s$ takes the value in (\ref{E:PKss}) or (\ref{E:Pss}), then the
degree~5 vertices have angle sums $\gamma=2\pi$ or $4\pi$,
respectively, regardless of the choices of $f_0$ in the
construction. I personally find this curious --- this persistence of
angle sums is nearly a packing criterion.

\vspace{10pt}
\subsection{Doyle Flowers} An early and fascinating chapter in the
story of circle packing involves a pattern for hex (degree 6) flowers
observed by Peter Doyle. We investigate this two-parameter family of
flowers here, but the interested reader can discover the beauty of the
``Doyle spirals'' that they lead to in \cite{BDS94}. In addition to
providing an obvious instance of a discrete exponential function,
these spirals raised the oldest question from the foundational period
that remains open, posed by Peter Doyle: {\sl Do there exist any circle
packings of the plane in which every circle has degree 6 other than
the familiar ``penny packing'', in which all circles share a common
radius, and the (coherent) Doyle spirals?}

The Doyle flowers involve two radius parameters, $a$ and $b$. If the
center $C$ is radius 1, then the petal radii take the form
\begin{equation} \label{E:doyleradii}
  \{a,b,b/a, 1/a,1/b,a/b\},\ \ a, b > 0.
\end{equation}
Remarkably, regardless of $a$ and $b$, this pattern of radii will
always form a 6-flower around $C$. More significant in the search for
schwarzians is the fact that the six triangular faces of that flower
fall into three similarity classes. For each $j=1,\cdots,6$, let $e_j$
be an edge, $f_j$ and $g_j$ the neighboring faces, and $\fp_j$ the
patch formed by their union. For each $j=1,2,3$, one can check that
there is a similarity $\Lambda:\fp_j\longrightarrow \fp_{j+3}$ with
$\Lambda(e_j)=e_{j+3}$. As a consequence, the sequence of schwarzians
takes the form
\begin{equation} \label{E:doylesch}
  \{s_1,s_2,s_3,s_1,s_2,s_3\}.
\end{equation}

View this pattern in the light of what we know about general
6-flowers: namely that $u_1=\fU_6(u_1,u_2,u_3)$; solving for
$u_3$ gives
\begin{equation} \label{E:doyle_eq}
  u_3=\dfrac{u_1+u_2}{3u_1u_2-1}.
\end{equation}
In other words, using $u_1,u_2$, we have a new 2-parameter
representation of the Doyle flowers in the space $\fF_6$.

It would be interesting to find the relationship between parameters
$a,b$ and $u_1,u_2$ (or $s_1,s_2$). However, there are more
challenging questions that the reader might like to take on. First,
sticking with the Doyle setting for a moment, note that the pattern of
a single Doyle flower propogates to an infinite spiral, all of whose
flowers share the identical set of schwarzians. The combinatorics
underlying all Doyle spirals is that of the hexagonal lattice $H$,
easily recognized as the planar lattice behind the penny
packing. Within $H$ are three families of parallel axes. What the
results above show is that, for a given Doyle spiral, all edges within
one of these family share the same schwarzian.

The challenge now is to conceive of other conditions on schwarzians
analogous to those of (\ref{E:doyle_eq}). What patterns, what families
of flowers, might emerge? In addition, are there patterns for flowers
that automatically propogate to larger, perhaps infinite,
configurations of circles? Examples might contribute to discrete
function theory as Doyle spirals have contributed discrete exponential
functions.

\vspace{10pt}
\subsection{Ring Lemma Flowers} \label{SS:ring}
In circle packing, the well known and important ``Ring Lemma''
provides a lower bound $c(n)$ for the ratio $r/R$ of petal radii $r$
to the center's radius $R$ in any univalent euclidean
$n$-flower. First introduced in \cite{RS87}, the extremal situations
and sharp constants were obtained in \cite{lH88} and were shown to be
reciprocal integers in \cite{dA97}. Of course, we are focused on
schwarzians not radii, so we work in our normalized setting. \F{ring}
suggests how the extremal normalized flowers develop in nested
fashion, with the extremal $n+1$-flower being obtained from the
extremal $n$-flower by adding the largest possible circle to its
smallest interstice. Continuing this {\sl ad infinitum}, we arrive at
what we might term an $\infty$-flower.

\vspace{10pt}
\begin{figure}[h]
\begin{center}
\begin{overpic}[width=.8\textwidth
  ]{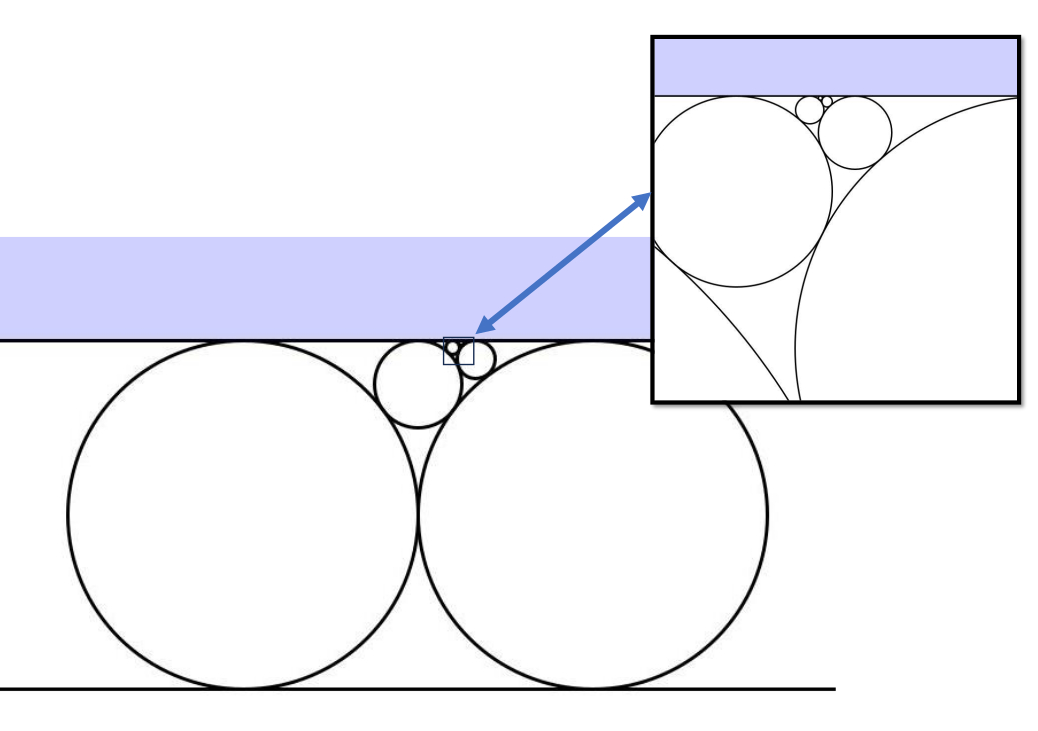}
  \put (5,42) {$C$}
  \put (5,0) {$c_0$}
  \put (11,16) {$c_1$}
  \put (37,32) {$c_2$}
  \put (24,39) {$0$}
  \put (57.5,39) {$2$}
  \put (43,52) {$t_{\infty}$}
  \linethickness{.6mm}
  \put (23.5,36) {\line(0,1){2.5}}
  \put (57.5,36) {\line(0,1){2.5}}
  \put (44,50) {\vector(0,-1){10}}
  \put (39.5,20) {$\bullet$}
\end{overpic}
\caption{Nested ``ring'' flowers}
\label{F:ring}
\end{center}
\end{figure}
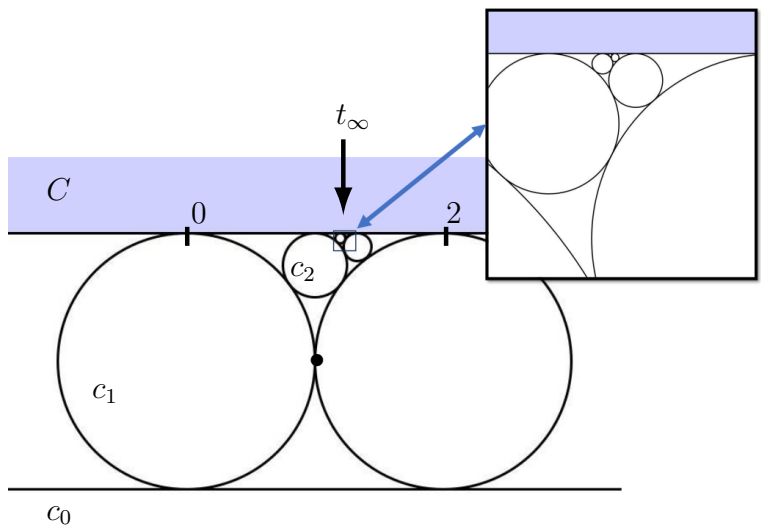

\vspace{10pt} At any stage in this development, the current flower is
rife with extraneous tangencies. Indeed, at a given stage we have an
$n$-flower whose smallest interstice is formed by $C$ and its two
smallest petals. When we plug the new petal into that
interstice to form an $n+1$-flower, the tangency between those two
petals becomes extraneous.

The packing's features allow us to compute precise schwarzians.
\F{ring_focus} focuses in on the interstice where a new petal, the
blue one, is being added. The red and green are the smallest previous
petals. Reindexing to accommodate the new petal, we assume the green
circle is $c_{j-1}$, the red is $c_{j+1}$ and the new blue is
$c_j$. There are extraneous tangencies, but nonetheless, functionally
$c_{j+1}$ is degree~4, $c_{j-1}$ is degree~5, and of course $c_j$ is
degree~3. (In alternating stages, the green would be on the right and
the red on the left.)

\begin{figure}[h]
\begin{center}
\begin{overpic}[width=.8\textwidth
  ]{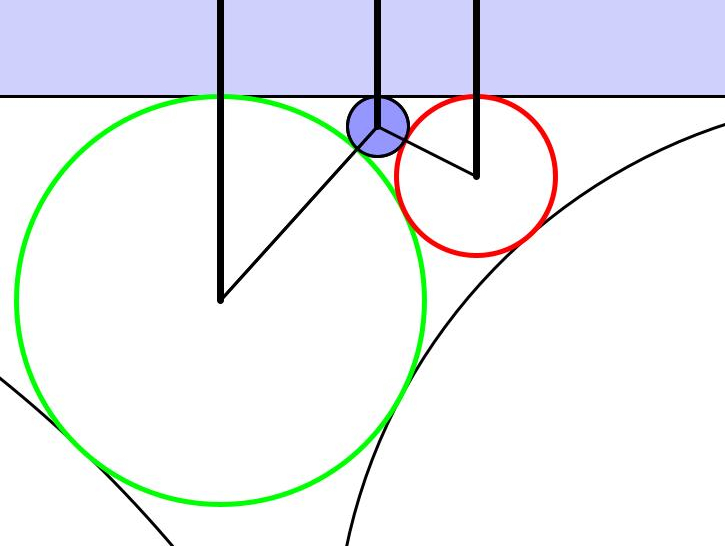}
  \put (5,68) {$C$}
  \put (32,70) {$e_{j-1}$}
  \put (54,70) {$e_{j}$}
  \put (67,70) {$e_{j+1}$}
  \put (60,55) {$e_r$}
  \put (37.5,47) {$e_g$}
  \put (50,39.5) {$e_b$}
  \put (23,33) {$c_{j-1}$}
  \put (68,50) {$c_{j+1}$}
  \linethickness{.65mm}
  \put (30,33.5) {\line(2.05,1.0){36}}
\end{overpic}
\caption{Inserting a new petal}
\label{F:ring_focus}
\end{center}
\end{figure}
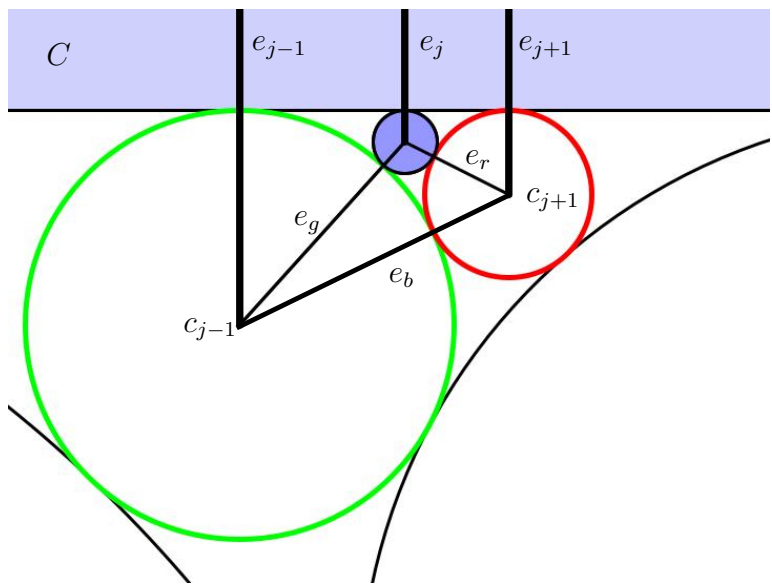

\vspace{10pt} Our local goal is the schwarzians $s_{j-1},s_j,$ and $s_{j+1}$
for the vertical edges $e_{j-1},e_j,$ and $e_{j+1}$, though we need the
schwarzians for the edges $e_r,e_g,e_x$ along the way. We will work in
the $u$-variables.

The blue petal, $c_j$, is degree~3, so Lemma~\ref{L:3f} gives
$u_j=u_r=u_g=1/\sqrt{3}$. The red petal, $c_{j+1}$, is degree~4 and
has edges $e_r$ and $e_x$. Lemma~\ref{L:4f} implies that $u_ru_x=2/3$,
so knowing that $u_r=1/\sqrt{3}$, we conclude that $u_x=2/\sqrt{3}$.
Finally, note that the green petal, $c_{j-1}$, has degree~5 and
successive edges $u_x,u_g,u_{j-1}$. Knowing $u_x$ and $u_g$ implies
$u_{j-1}=\fU_5(u_x,u_g)=\sqrt{3}$.  Summarizing for the target schwarzians,
we conclude
\begin{align} \label{E:sss}
  &s_{j-1}=1-\sqrt{3}\sim -0.732051,\notag \\
  &s_j=1-1/\sqrt{3}\sim 0.422650,\\
  &s_{j+1}=1-2/\sqrt{3}\sim -0.154701.\notag
\end{align}

So, what do we conclude about the schwarzians of a full flower?
Observe that when a new petal is added in our construction, it
converts its smaller neighbor, degree~3 in the previous step, to
degree~4, while it converts its larger neighbor to degree~5. That
larger neighbor, $c_{j-1}$ in \F{ring_focus}, remains unchanged
thereafter, so its schwarzian remains at $1-\sqrt{3}$. On the other
hand, the schwarzian for the smaller neighbor, $c_{j+1}$, is only
temporary, as it will change with the next added petal. So here is the
typical sequence for a Ring Lemma $n$-flower, stated in the alternate
$u_{\cdot}$-variables:
\begin{equation} \notag\label{E:ring_u}
  \{\sqrt{3},\ \cdots\cdots,\ \sqrt{3},
  \ \frac{1}{\sqrt{3}},\ \frac{2}{\sqrt{3}},
  \ \sqrt{3},\ \cdots\cdots,\ \sqrt{3},\ \frac{1}{\sqrt{3}}\}.
\end{equation}
With every increase in $n$, the $2/\sqrt{3}$ will convert to
$\sqrt{3}$, the $1/\sqrt{3}$ will convert to $2/\sqrt{3}$, and a new
$1/\sqrt{3}$ will be inserted between them.

\vspace{10pt} Past experience with the Ring Lemma suggests that one
should not leave these flowers without looking around for interesting
numerical features.  In \cite{dA97} and \cite{AS97} the Fibonacci
sequence, the Descartes Circle Theorem, and the golden ratio all play
significant roles. In our normalized setting, we can add Farey numbers
to that list.

\vspace{10pt} So, let's look around! As visually suggested in
\F{ring_focus}, the local picture around a new petal has a static
asymptotic limit. We've seen that $u_j=1-s_j=1/\sqrt{3}$, so applying
(\ref{E:qr}) and adjusting the indexing we see this recurrence
relation among the radii:
\begin{equation} \notag
  \dfrac{1}{\sqrt{r_{j+1}}}=\dfrac{1}{\sqrt{r_j}}+\dfrac{1}{\sqrt{r_{j-1}}}.
\end{equation}
This is a generalized Fibonacci pattern and is precisely the
recurrence observed in \cite[\S4]{AS97}. As there, one can conclude
that
\begin{equation} \notag\label{E:ratio}
  \dfrac{r}{r'}\longrightarrow (\dfrac{1+\sqrt{5}}{2})^2=\tau^2,
\end{equation}
where $r'$ is the radius of a new petal, $r$ is the radius of the
previous new petal, and $\tau$ is the famous Golden Ratio.

How do Farey numbers enter the picture? {\bf Caution:} for this
discussion we must scale our normalized Ring Lemma flowers by
1/2. Thus the tangency points $t_j$ and radii $r_j$ are now scaled by
1/2, putting all the tangency points in $[0,1]$.

\vspace{5pt} One can deduce from the Descartes Circle Theorem, that if
a circle is placed in the interstice of three mutually tangent circles
whose bends (reciprocal radii in the terminology of F.~Soddy
\cite{fS36a}) are integers, then that circle's bend will also be an
integer. In our construction, we continually put new circles in
interstices. One can prove inductively that all radii (after our
scaling by 1/2) are reciprocal integers. From this one can conclude
that all the tangency points $t_j$ are rational numbers. Indeed, these
all fall into what's known as the ``Farey sequence'' in $[0,1]$ and
are subject to the counterintuitive Farey arithmetic. Consider a
tangency $t_j$ for a new circle in our construction, between the
tangency points $t$ and $t'$ for the previous two new circles. We may
write $t=a/b$ and $t'=a'/b'$ as rational numbers in lowest terms. From
the Descartes Circle Theorem, one can show that
\begin{equation} \label{E:tpat}
  t_j=\dfrac{a+a'}{b+b'}.
\end{equation}
To see the overall pattern of (rescaled) tangency points, we will
redefine the indexing as a sequence $\{t_j\}$. Here $t_0=0$, $t_1=1$,
and thereafter, let $t_j$ denote the tangency point of the next new
petal added, so $t_j$ always falls between $t_{j-1}$ and
$t_{j+1}$. (This indexing is not that used for individual
$n$-flowers.)  Now write $t_0=0/1$ and $t_1=1/1$ and then repeatedly
apply (\ref{E:tpat}). (There is one choice involved; after $t_2=1/2$
in \F{ring} we chose $t_3=\frac23$ rather than $t_3=\frac13$. ) Here,
then, are the first few values
\begin{equation} \notag
  \{\frac01,\ \frac11,\ \frac12,\ \frac23,\ \frac35,\ 
  \frac58,\ \frac8{13},\ \frac{13}{21},\ \frac{21}{34}\ \cdots\cdots\}.
\end{equation}
As one can see, $t_j=\mF_j/\mF_{j+1},j=1,\cdots$, where $\mF_j$ is the $j$th
Fibonacci number.  It is well known that this ratio converges to
$1/\tau$. In other words, the new petals in the infinite flower
suggested by \F{ring} converge to the point $t_{\infty}=2/\tau$.

In conclusion: Circles --- aren't they grand!

\appendix \label{S:appendix}
\section{}

Here we detail various computations with intrinsic schwarzians. We
start with four situations underlying our construction of normalized
flowers from given schwarzians. Next we compute the schwarzians for
uniform flowers. Finally we work out the relationship between a
Schwarzian derivative and the intrinsic schwarzians of domain and
range.

\subsection{Layout Computations} \label{SS:A1}
We work with flowers in their normalized positions; see \F{f7}) for
the notation. Note that $C$ and $c_0$ are tangent at infinity, so the
imaginary axis represents the edge $e_0$ between them, associated with
schwarzian $s_0$.

In computing the remaining petals, we encounter three situations, and
possibly a fourth if there is branching. In each there is an edge $e$
of interest connecting the upper half plane to a petal circle (the shaded one)
whose position has already been established.  There is also an
``initial'' neighboring petal (green) which is also in place. Our task
is to find data for the companion ``target'' petal (red), that which
is opposite to the initial petal across edge $e$. The shaded face $f$
is that formed by the central circle, the shaded circle, and the
initial circle. We are given the initial data for the two petals in place
and the schwarzian $s$ for $e$ and show the computations of data for
the target circle; in particular, we compute its tangency point $t$
and its radius $r$. The formulas we arrive at are easier to work with
if we introduce $u=1-s$ as an alternate to the variable $s$
itself. Situations~1-3 are illustrated as they would appear in
un-branched flowers. The computations, however, are fully general as
we discuss in Situation~4.

\vspace{10pt}
\noindent{\bf Situation 1.}
We begin with the edge $e=e_0$, connecting the two half planes as
illustrated in \F{sit1}. The petal $c_0$ (a half plane)
and the initial petal $c_1$ are in place as part of our
normalization. The target is the clockwise neighbor of $c_0$, namely,
the petal $c_{n-1}$. Being tangent to both half planes, its radius is
$r_{n-1}=1$ and we need only compute its tangency point $t_{n-1}$ from
the schwarzian $s_0$.

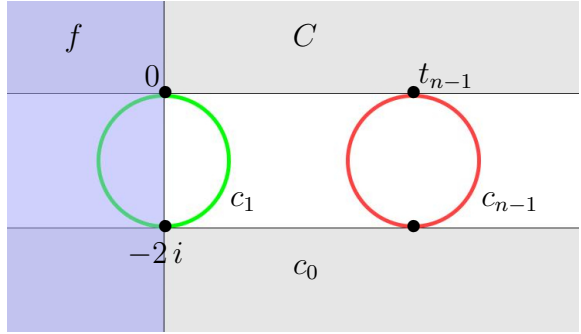
\begin{figure}
\begin{center}
\begin{overpic}[width=.6\textwidth
  ]{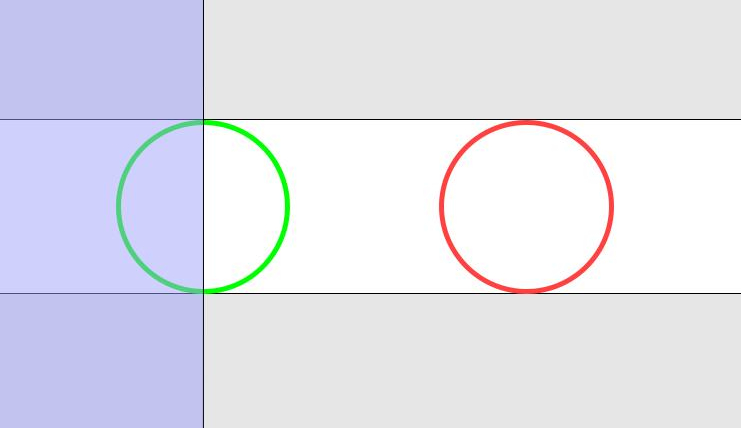}
\put (10,50) {$f$}
\put (50,50) {$C$}
\put (50,10) {$c_0$}
\put (24,43) {0}
\put (21,12) {$-2\,i$}
\put (26.25,40.5) {$\bullet$}
\put (26.25,17) {$\bullet$}
\put (39,22) {$c_1$}
\put (83,22) {$c_{n-1}$}
\put (72,43.5) {$t_{n-1}$}
\put (69.75,40.5) {$\bullet$}
\put (69.75,17) {$\bullet$}
\end{overpic}
\caption{Situation~1: layout the red circle}
\label{F:sit1}
\end{center}
\end{figure}

\vspace{10pt} Let $s=s_0$. The computations involve the M\"obius
transformation $M_s$ (see (\ref{E:mf})) and the M\"obius $m_f$ mapping
the points $\{1,\omega,\omega^2\}$ to $\{\infty,-2i,0\}$, and hence
mapping $C_v$ to the upper half plane.
\begin{align}\notag
m_f\circ M_s^{-1}=&  \begin{bmatrix}2i & {\ } & -\sqrt{3}+i\\
   {\ } & {\ }\\
    -1/2+\sqrt{3}/2\,i & {\ } &
    1/2-\sqrt{3}/2\,i
\end{bmatrix}
\begin{bmatrix}
  1-s & s \\ {\ } & {\ }\\-s & 1+s
\end{bmatrix} \\
\notag{\ } & \\
= &\begin{bmatrix}\notag \sqrt{3}s+(2-3s)\,i & {\ } &
  -\sqrt{3}(1+s)+(3s+1)\,i\\
   {\ } & {\ }\\
    -1/2+\sqrt{3}/2\,i & {\ } &
    1/2-\sqrt{3}/2\,i
  \end{bmatrix}
\end{align}

\vspace{10pt}
\noindent Applying this transformation to $C_b$ gives the normalized
petal $c_{n-1}$. In particular, applying it to the tangency point
$(5-\sqrt{3}\,i)/2$ in the base patch $\fp_{\Delta}$ yields the
normalize tangency point $t_{n-1}$, expressed using $u_0=1-s_0$ as
\begin{equation} \tag{S1}\label{E:S1}
  t_{n-1}=2\sqrt{3}\,u_0\qquad\text{ and }\qquad r_{n-1}=1.
\end{equation}

\vspace{10pt}
\noindent{\bf Situation 2.} 
We move now to the edge $e=e_1$ with the target being $c_2$. The
relevant schwarzian is $s=s_1$ and the initial petal is the half plane
$c_0$, green in \F{sit2}.

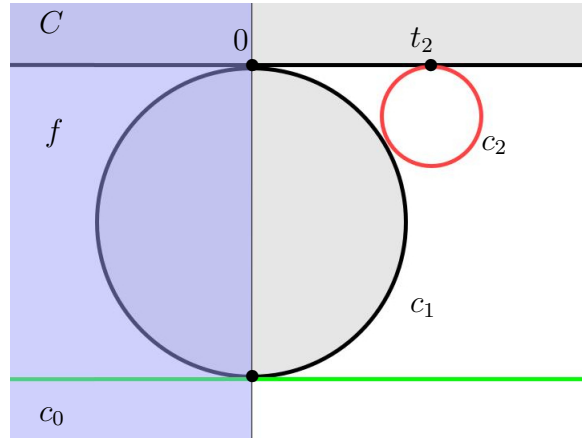
\begin{figure}[h]
\begin{center} 
\begin{overpic}[width=.6\textwidth
  ]{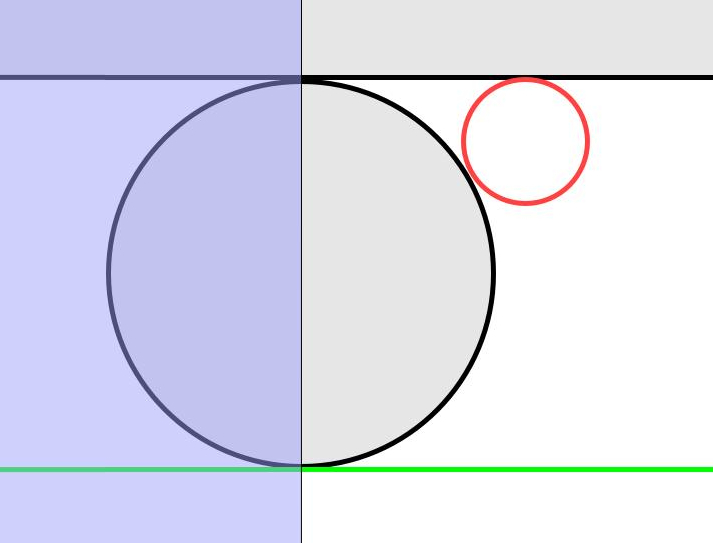}
\put (6,52) {$f$}
\put (5,71) {$C$}
\put (5,3) {$c_0$}
\put (39,68) {0}
\put (41,64) {$\bullet$}
\put (41,9.5) {$\bullet$}
\put (70,22) {$c_1$}
\put (82.5,50.5) {$c_2$}
\put (70,68.5) {$t_2$}
\put (72.25,64) {$\bullet$}
\end{overpic}
\caption{Situation~2: layout the red circle}
\label{F:sit2}
\end{center}
\end{figure}

\vspace{10pt}
We proceed by modifying the previous argument.  The shaded face $f$ is
the same, but the M\"obius $m_f$ must now map $\{1,\omega,\omega^2\}$
to $\{0,\infty,-2i\}$. We accomplish this by precomposing the earlier
$m_f$ with a rotation by $\omega^2$. The result is
\begin{align}\notag
  m_f\circ M_s^{-1}=&  \begin{bmatrix}\sqrt{3}-i & {\ }
    & -\sqrt{3}+i\\
    1 & {\ } &
    \dfrac{1}{2}-\dfrac{\sqrt{3}}{2}i
\end{bmatrix}
\begin{bmatrix}
  1-s & s \\ {\ } & {\ }\\-s & 1+s
\end{bmatrix} \\
\notag{\ } & \\
= &\begin{bmatrix}\notag \sqrt{3}-i & {\ } &
  -\sqrt{3}+i\\
  {\ } & {\ }\\
    1-3s/2 +(\sqrt{3}s/2)i & {\ } &
    (1+3s)/2-(\sqrt{3}(1+s)/2)i
  \end{bmatrix}
\end{align}

\vspace{10pt}
\noindent Applying this transformation to $C_b$ gives the normalized
petal $c_1$. Note that $m_f$ now maps $C_w$ to the upper half plane,
so applying the above M\"obius to the tangency point
$(5+\sqrt{3}\,i)/2$ in the base patch yields the displacement to the
normalized tangency point $t_2$; simple geometric computations give
the radius. We use the variable $u_1=1-s_1$.
\begin{equation} \tag{S2}\label{E:S2}
  t_2=2/(\sqrt{3}\,u_1)\qquad\text{ and }
  \qquad r_2=(t_2)^2/4=1/(\sqrt{3}\,u_1)^2.
\end{equation}

\vspace{10pt}
\noindent{\bf Situation 3.} 
We are left to treat the generic situation suggested by \F{sit3}. The
edge $e$ goes from the central circle $C$ to the shaded circle, with
its schwarzian $s$ and variable $u=1-s$.  (Note that the half plane
for $c_0$ is no longer necessarily involved.) We assume the shaded
circle has radius $R$, while the initial green circle has radius
$r$. It is convenient to position the shaded circle tangent to $C$ at
the origin, and then our goal is to compute the tangency point
$\delta$ (the displacement from 0) and the radius $\rho$ of the
red target circle.

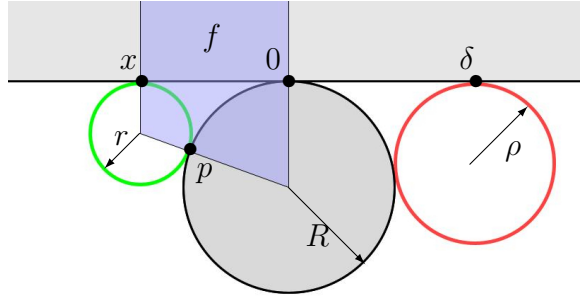
\begin{figure}[h]
\begin{center}
\begin{overpic}[width=.6\textwidth
  ]{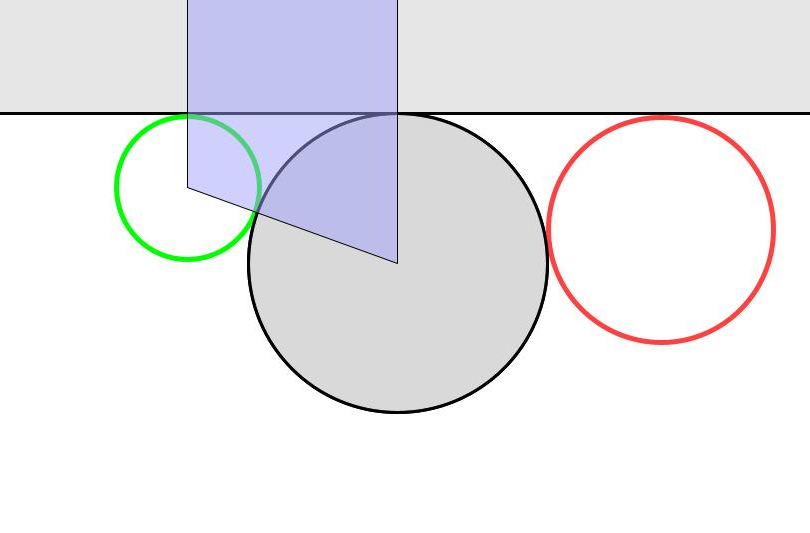}
\put (23,42.75){\vector(-2.5,-2.5){6.25}}
\put (49,33.25){\vector(2.5,-2.5){13.5}}
\put (80.75,37.25) {\vector(2.5,2.5){10.25}}
\put (18.5,40.75) {$r$}
\put (52,23) {$R$}
\put (87,39) {$\rho$}
\put (45,54) {0}
\put (34,58) {$f$}
\put (19.5,54) {$x$}
\put (79,54) {$\delta$}
\put (33,35) {$p$}
\put (22,50.5) {$\bullet$}
\put (47.75,50.5) {$\bullet$}
\put (80.5,50.5) {$\bullet$}
\put (30.5,38.5) {$\bullet$}
\end{overpic}
\caption{Situation~3: layout the red circle}
\label{F:sit3}
\end{center}
\end{figure}

Elementary geometric computations yield:
\begin{equation} \notag
  x=-2\sqrt{rR},\qquad\qquad
  p=-(\dfrac{2R}{R+r})(\sqrt{rR}+(r+R)\,i).
\end{equation}
The following M\"obius transformation $m$ will convert this generic
situation to Situation~2. Namely, $m$ maps $\{x,p,0\}$ to
$\{\infty,-2i,0\}$, so the configuration of \F{sit3} morphs into
that of \F{sit2}. 
  \begin{equation} \notag
    m=\begin{bmatrix}
    1+\sqrt{r/R\,\,}\,i & {\ } &0 \\
    {\ } & {\ } &\\
      (\sqrt{R/r\,\,}+i)/2 &{\ } &
      R+\sqrt{rR\,}\,i
    \end{bmatrix}
    \end{equation}
  The tangency point $t_2$ in \F{sit2} corresponds to the tangency
  point $\delta$ in \F{sit3}, so $\delta$ is obtained by applying
  $m^{-1}$ to $t_2$. An annoying calculation gives, in the alternate
  variable $u$,
  \begin{equation} \tag{S3}\label{E:S3}
    \delta(u,r,R)=\dfrac{2R}{(\sqrt{3}\,u-\sqrt{R/r\,\,})}
    \text{ and }\quad
    \rho=\dfrac{1}{(\sqrt{3}\,u/\sqrt{R}-1/\sqrt{r})^2}.
  \end{equation}

We will also need to reverse these computations in a particular
situation in order to compute $s$. The situation is this: the values
$r$ and $R$ are known, $\delta$ is positive, and the computed radius $\rho$
comes out to be $1$. What is $s$? We compute $u$, then $s=1-u$.
\begin{equation} \tag{R3}\label{E:R3}
  \text{When $R,r$ are known, $\delta>0$, and $\rho=1$: }\qquad
  u=\dfrac{\sqrt{R}+\sqrt{R/r}}{\sqrt{3}}.
\end{equation}

Situation~2 is the
limiting case of Situation~3 when $r$ grows to $\infty$, so
(\ref{E:S2}) follows from (\ref{E:S3}). Also, note that when applying
(\ref{E:S3}), the quantity $\delta$, which represents the displacement
of the target circle from its shaded neighbor, can be zero or
negative. An example is the branched flower of \F{four}(c): with
initial circle $c_2$, the displacement of the target $c_4$ from $c_3$
is negative. This puts us in the following branching situation.

\vspace{10pt}
\noindent{\bf Situation 4.} Branching is initiated during a layout
step if and only if (\ref{E:S3}) results in a displacement $\delta\le 0$.
\F{br1}(a) illustrates the most typical case, with
$\delta_j=(t_{j+1}-t_j)<0$. However, it is laying out the next circle
that we need to concentrate on, as shown in \F{br1}(b).  (The color
codings are as before; known green and shaded petals in place , a red
target petal to be positioned based on the schwarzian of the edge to
the shaded circle.)

\vspace{10pt}
\begin{figure}[h]
\begin{center}
\begin{overpic}[width=.9\textwidth
  ]{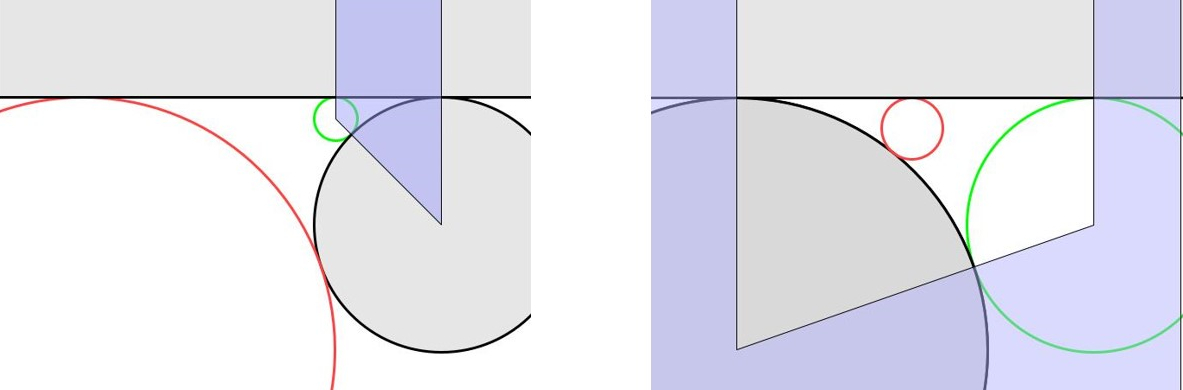}
  \put (32,29) {$f$}
  \put (95,4) {$f$}
  \put (32,14) {$c_j$}
  \put (88,14) {$c_j$}
  \put (24.5,22) {$r$}
  \put (28.3,23) {\vector(-1,-1.3) {1.42}}
  \put (78,22.2) {$\rho$}
  \put (77.15,22.4) {\vector(1,-1.3) {1.95}}
  \put (12,11.5) {$\rho$}
  \put (7.5,3) {\vector(1,1){15}}
  \put (68,11.5) {$R$}
  \put (62.5,3.5) {\vector(1,.8){16.2}}
  \put (35.8,8) {$R$}
  \put (37.4,13.8) {\vector(-1,-1.6){5.6}}
  \put (90,8) {$r$}
  \put (92.4,13.8) {\vector(-1,-1.6){5.6}}
  \put (7.45,24.5) {\line(0,1){5}}
  \put (37.1,24.5) {\line(0,1){5}}
  \put (37.2,24.5) {\line(0,1){5}}
  \put (37.3,24.5) {\line(0,1){5}}
  \put (37.4,24.5) {\line(0,1){5}}
  \put (37.2,27) {\vector(-1,0){30}}
  \put (19,28.5) {$\delta=\delta_{j}$}
  \put (62.3,24.5) {\line(0,1){5}}
  \put (62.2,24.5) {\line(0,1){5}}
  \put (62.4,24.5) {\line(0,1){5}}
  \put (77.2,24.5) {\line(0,1){5}}
  \put (62.2,27) {\vector(1,0){15.4}}
  \put (65,28.5) {$\delta=\delta_{j+1}$}
  \put (1,4) {$c_{j+1}$}
  \put (56,4) {$c_{j+1}$}
  \put (-5,30) {(a)}
  \put (50.5,30) {(b)}
\end{overpic}
\caption{Situation~4: layout the red circle}
\label{F:br1}
\end{center}
\end{figure}
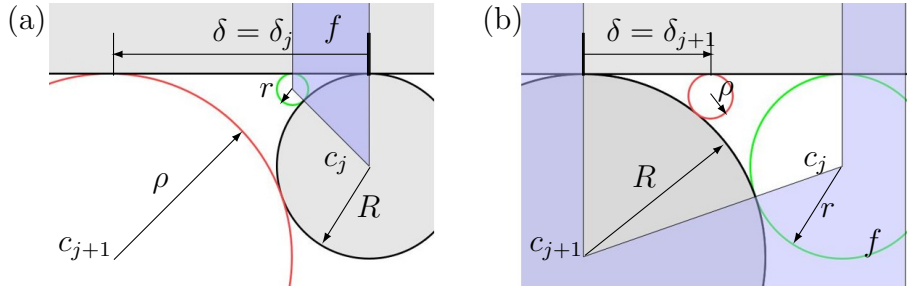
By the formula in (\ref{E:S3}), when $\delta<0$, then
$(\sqrt{3}u/\sqrt{R}-1/\sqrt{r})<0$. This means, in turn, that our
previous expression $1/\sqrt{\rho}= (\sqrt{3}u/\sqrt{R}-1)$ is no
longer true, as it requires absolute values on the right hand
side. Subsequent formulas like those in (\ref{E:sqrs}) and
(\ref{E:qr}) fail, and ultimately, $\fU_n$ is no longer
represented in a closed formula. This is what makes branched flowers
more difficult to manipulate.

\F{br1}(b) is typical of what we refer to as Situation~4. Notice that
the new displacement, $\delta_{j+1}$, is again in the positive
direction. The computations require a modification of (\ref{E:S3}).

\begin{align} \label{E:S4}
  &\text{When the {\bf previous} displacement was negative, then
  (\ref{E:S3}) becomes}\notag\\
  &\qquad\delta(u,r,R)=\dfrac{2R}{(\sqrt{3}\,u+\sqrt{R/r\,\,})}
    \text{ and }\quad
    \rho=\dfrac{1}{(\sqrt{3}\,u/\sqrt{R}+1/\sqrt{r})^2}.\tag{S4}
\end{align}

\noindent
The ``previous'' step refers to that where $R$ was computed.  {\sl
  Apropros} to our earlier comments, the modification here is simply
replacing $\sqrt{R}$ by $-\sqrt{R}$ in (\ref{E:S3}).  (There is one
other detail: the standing assumption $u\ge0$ is also required
to ensure that this new displacement $\delta$ is positive).

Another possiblity leading to branching is pictured in \F{hp}.
Namely, when $(\sqrt{3}\,u-1/\sqrt{R/r})=0$ in (\ref{E:S3}),
so $\delta$ is undefined. In essence, $\delta=\infty$, $R=\infty$,
and the petal $c_{j+1}$ is a half plane (i.e., tangent to $C$ at $\infty$).
\F{hp} illustrates the situation when placing the next petal
$c_{j+2}$, which necessarily has the same radius $r_j$ as $c_{j}$.
For its tangency point, note that \F{hp} is a version of \F{sit1}.
Applying (\ref{E:S1}), scaling by $r_j$, and taking the order
$t_j,\infty,t_{j+2}$ of the tangencies about $C$ into account,
we have $t_{j+2}-t_j=-2\sqrt{3}u_{j+1}r_j$.

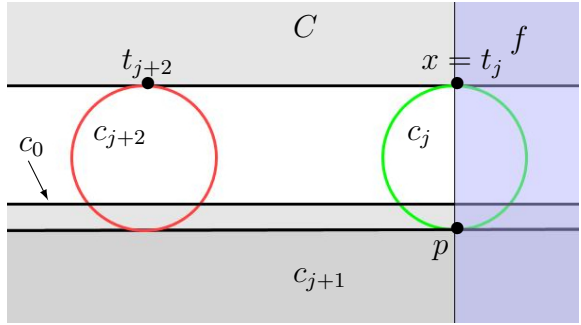
\begin{figure}[h]
\begin{center}
\begin{overpic}[width=.6\textwidth
  ]{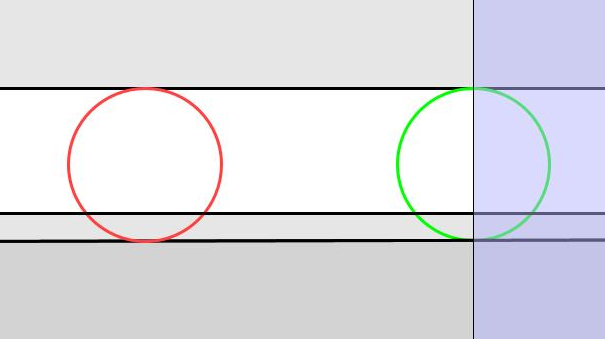}
\put (88,48) {$f$}
\put (50,50) {$C$}
\put (15,32) {$c_{j+2}$}
\put (70,32) {$c_{j}$}
\put (23.25,40.5) {$\bullet$}
\put (20,44.3) {$t_{j+2}$}
\put (50,7) {$c_{j+1}$}
\put (72.5,44.3) {$x=t_j$}
\put (77.4,40.5) {$\bullet$}
\put (77.4,15.2) {$\bullet$}
\put (74.5,12) {$p$}
\put (2,30) {$c_0$}
\put (3.65,28.3) {\vector(1,-2){3.5}}
\end{overpic}
\caption{Laying out the red circle when $c_j$ ia a half plane}
\label{F:hp}
\end{center}
\end{figure}

\vspace{10pt} We conclude this subsection by explaining the two
exceptions to successful layout as listed in Theorem~\ref{T:LP}. The
exceptional situations occur when $j+1=n-2$ in \F{br1} and
\F{hp}. Regarding exception (a), if \F{hp} occurs (so $c_{j+1}$ is the
penultimate petal $c_{n-2}$) then the Layout Process fails because
placing $c_{j+2}$ (i.e., $c_{n-1}$) with mandated radius 1 is either
impossible (if $r_{j}\not=1$) or ambiguous (since $u_{n-2}$ is
unknown). Regarding exception (b), look to \F{br1}(b). Though $c_1$ is
not pictured here, if the tangency point of the red circle, $t_{n-1}$,
is negative (to the left of $t_1=0$), then (\ref{E:S1}) implies $u_0$
is negative, that is $s_0>1$, which is not allowed.

\vspace{10pt}
\subsection{\bf Uniform Petals} \label{SS:A2}
The schwarzians for a uniform $n$-flower take a constant value that we
have labeled $\fs_{n}$. Here we show that 
\begin{equation}\label{E:S4}
  \fs_n=1-\dfrac{2\cos(\pi/n)}{\sqrt{3}},\ \ n\ge 3.
\end{equation}
We will base our computations on \F{unif_sch}, with $C$ the unit
circle and successive petals $c_{n-1},c_0,c_1$ sharing a common radius. Our
interest is in the schwarzian $s$ for the edge from $C$ to $c_0$.

\begin{figure}[h]
\begin{center}
\begin{overpic}[width=.6\textwidth
  ]{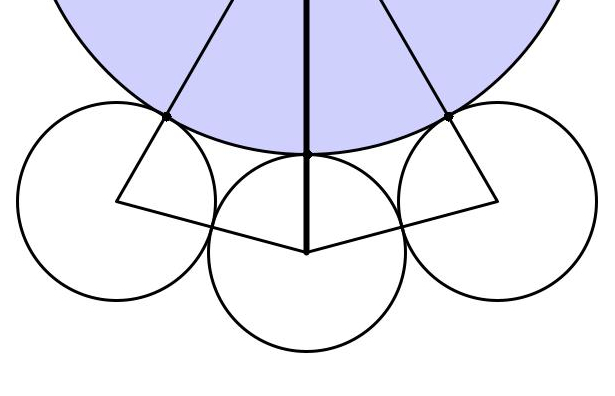}
\put (45.5,35.5) {-i}
\put (21,58) {$C$}
\put (10.5,22) {$c_{n-1}$}
\put (42,13.5) {$c_0$}
\put (73,22) {$c_1$}
\put (20.75,43) {$\zeta$}
\put (75,43) {$\xi$}
\linethickness{.5mm}
\put (56,64) {\line(1,-4){9}}
\put (64,26.5) {$\bullet$}
\put (48.5,38.5) {$\bullet$}
\put (71.75,44.5) {$\bullet$}
\put (26,44.5) {$\bullet$}
\put (64,48) {$f$}
\put (52.5,52) {$\alpha$}
\put (62,24) {$p$}
\end{overpic}
\caption{Uniform Petals: compute the schwarzian}
\label{F:unif_sch}
\end{center}
\end{figure}
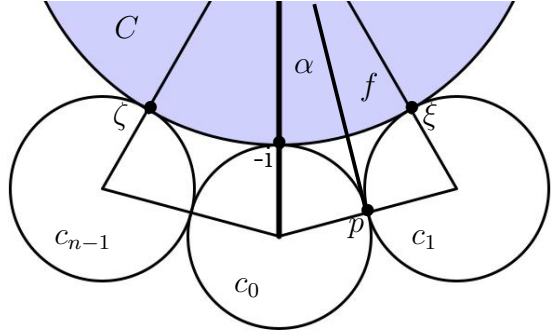

\vspace{10pt} Let the angle $\alpha$ be one half of the angle at the
origin in face $f$ formed by the triple $\{C,c_0,c_1\}$.  If these
petals were taken from a uniform $n$-flower, then $\alpha=\pi/n$.
However, the following computation works for any $\alpha,
0<\alpha<\pi/2$. Note that the tangency points of the circles are
\begin{equation} \notag
  \xi=\sin(2\alpha)-i\,\cos(2\alpha)\ \ \ \text{ and }
  \ \ \ \zeta=-\sin(2\alpha)-i\,\cos(2\alpha).
\end{equation}
Let $T$ denote the M\"obius transformation which maps $\{-i,p,\xi\}$
to $\{\infty,-2i,0\}$, where $p$ is the tangency point between $c_0$
and $c_1$. This transformation puts the four circles in the standard
normalized positions as they appear in \F{sit1}. In particular,
$T(\zeta)$ is the tangency point labeled $t_{n-1}$ there. Applying
(\ref{E:S1}), we conclude that $1-s=T(\zeta)/(2\sqrt{3})$. I will
leave the computation of $T$ to the curious reader, but here's the
general result:
\begin{equation} \label{E:alpha}
  s=1-\dfrac{2\cos(\alpha)}{\sqrt{3}},\ 0<\alpha<\pi/2.
\end{equation}

\subsection{\bf Special Computations} \label{SS:A3}
We outline two computations referred to in \S\ref{S:class}. These are
similar in nature and, though elementary with a symbolic math package,
provide great fun via pencil-and-paper. Both involve the restriction
of a discrete mapping $F$ between circle packings to a domain patch
$\fp$ and its image patch $\fp'=F(\fp)$. The related objects involved
are the edges $e,e'$, their intrinsic schwarzians $s,s'$, their
tangency points $t,t'$, the face mappings $m_f:f\longrightarrow f'$,
$m_g:g\longrightarrow g'$, and the discrete Schwarzian derivative
$\sigma=\Sigma_F(e)$. Both situations also involve a M\"obius
transformation $m(z)=(az+b)/(cz+d)$; we write $m$ in matrix form
\begin{equation} \notag
  m=\begin{bmatrix} a & b\\ c & d \end{bmatrix},
  \ \ \text{ with }\ ad-bc=1.
\end{equation}

The first computation relates the Schwarzian derivative and
the two intrinsic schwarzians $s,s'$. The Schwarzian
derivative $\sigma$ arises in
\begin{equation}\notag
  m_g^{-1}\circ m_f=\bI+\sigma\begin{bmatrix}t & -t^2\\1 & -t\end{bmatrix}.
\end{equation}
For the intrinsic schwarzians we need to identify these additional
M\"obius transformations identifying faces:
\begin{align} \notag
  &\text{For }\fp=\{f\,|\,g\}\,\text{:}
    \qquad\quad \mu_f:f_{\Delta}\longrightarrow f;
    \ \ \ \ \mu_g:g_{\Delta}\longrightarrow g.\notag\\
  &\text{For }\fp'=\{f'\,|\,g'\}\,\text{:}
    \quad\quad \ \nu_{f}:f_{\Delta}\longrightarrow f';
    \ \ \ \ \nu_g:g_{\Delta}\longrightarrow g'.\notag
\end{align}

\vspace{10pt}
Manipulating the expression for schwarzians and taking $m=\mu_f$, we get
\begin{equation} \notag
    \nu_g^{-1}\circ \nu_f=\mu_g^{-1}\circ (m_g^{-1}\circ m_f)\circ m
    \ \ \text{ and }
\end{equation}
\begin{equation} \notag
    \mu_g^{-1}=\begin{bmatrix}1+s & -s \\ s & 1-s\end{bmatrix}
    \begin{bmatrix}d & -b\\-c & a\end{bmatrix}.
\end{equation}
Putting these into matrix form gives
\begin{equation} \notag
  \begin{bmatrix}1+s' & -s'\\ s' & 1-s'\end{bmatrix}=
    \begin{bmatrix}1+s & -s\\s & 1-s\end{bmatrix}
      \begin{bmatrix}d & -b\\-c & a\end{bmatrix}
        \begin{bmatrix}1+\sigma t & -\sigma t^2\\\sigma & 1-\sigma t\end{bmatrix}
          \begin{bmatrix}a & \ b\\c & d\end{bmatrix}.
\end{equation}
The many pleasant surprises in a pencil-and-paper simplification yields
\begin{equation} \notag
  \begin{bmatrix}1+s' & -s'\\ s' & 1-s'\end{bmatrix}=\bI+
    \begin{bmatrix} s+\sigma/(c+d)^2 & -(s+\sigma/(c+d)^2)\\
      s+\sigma/(c+d)^2 & -(s+\sigma/(c+d)^2)\end{bmatrix},
\end{equation}
implying $s'=s+\sigma/(c+d)^2$. Moreover, the expression on the right
is associated with the map of $\fp_{\Delta}\longrightarrow
\fp'$ and with the tangency point $\tau= 1$ in its domain. The
Schwarzian derivative $s'=s+\sigma/(c+d)^2$ may therefore be rewritten
\begin{equation}
	s'=s+\Sigma_F(e)\cdot m'(1).
\end{equation}
(As a side note, $\Sigma_F(e)\cdot m'(1)$ is real.)

\vspace{10pt} Schwarzian derivatives --- both classical and discrete
--- are unchanged under post-composition by M\"obius
transformations. Our second computation derives the chain rule for
discrete Schwarzian derivatives under pre-composition.  We will rely
on the notations above, except that $m$ now respresents an arbitrary
M\"obius transformation and the base patch $\fp_{\Delta}$ is replaced
by the patch $\fp''=m^{-1}(\fp)=\{f''\,|\,g''\}$ with its tangency
point denoted $\tau$.

We start with the function $F:\fp\longrightarrow \fp'$.  Its
Schwarzian derivative $\sigma=\Sigma_F(e)$ is derived from the
expression
\begin{equation} \notag
  m_g^{-1}\circ m_f=\bI+\sigma \begin{bmatrix}t & -t^2\\
    1 & -t\end{bmatrix} =\begin{bmatrix}1+\sigma t & 1-\sigma t^2\\
    \sigma & 1-\sigma t\end{bmatrix}.
\end{equation}

The issue is, given $m$, what is the Schwarzian derivative for $F\circ
m:\fp''\longrightarrow \fp'$, denoted $\Sigma_{F\circ m}(e'')$?  This is
derived from $\nu_g^{-1}\circ \nu_f$, involving the face maps
$\nu_f:f''\longrightarrow f'$ and $\nu_g:g''\longrightarrow g'$.  Note
that $\nu_f=m_f\circ m,\nu_g=m_g\circ m$. Therefore,
\begin{equation} \notag
  \nu_g^{-1}\circ \nu_f=(m_g\circ m)^{-1}\circ m_f\circ m
  = m^{-1}\circ (m_g^{-1}\circ m_f)\circ m.
\end{equation}
Manipulating this, we arrive at
\begin{align} \notag
\nu_g^{-1}\circ \nu_f&=m^{-1}\cdot\bigl[\bI+\sigma
  \begin{bmatrix} t & -t^2 \\ 1 & -t \end{bmatrix} \bigr]\cdot m\notag \\
&=\bI+\sigma\begin{bmatrix} d & -b \\ -c & a \end{bmatrix}
\begin{bmatrix} t & -t^2 \\ 1 & -t \end{bmatrix}
  \begin{bmatrix} a & b \\ c & d \end{bmatrix}.\notag
\end{align}
Since $m$ identifies $e''$ with $e$, we have $m(\tau)=t$. Using this
to replace $t$ and enjoying further pencil-and-paper work, one arrives
at
\begin{equation} \label{E:almost}
  \nu_g^{-1}\circ \nu_f=\bI+\dfrac{\sigma}{(c\tau+d)^2}
    \begin{bmatrix}\tau & -\tau^2\\ 1 & -\tau\end{bmatrix}.
\end{equation}
This gives our discrete chain rule, which is placed here beside the
classical version:
\begin{align} \label{E:SSigma}
  &\Sigma_{F\circ m}(e'')=\sigma/(c\tau+d)^2=
  \Sigma_F(m(e''))\cdot m'(\tau)\\
  &S_{\phi\circ m}(z)=S_{\phi}(m(z))/(cz+d)^4=
  S_{\phi}(m(z))\cdot (m'(z))^2.\notag
\end{align}
These diverge in that the discrete version involves $m'$ rather than
$(m')^2$. The author has no concrete explanation for this
difference. It is perhaps worth noting, however, that for mappings
between circle packings, the ratios of image radii to domain radii
serves as a proxy for the absolute value of the classical derivative;
see, for example, \cite{CDS03}. In some sense, these mappings already
incorporate a derivative, and this may subtly influence this chain
rule.


\bibliographystyle{amsplain}
\bibliography{SchwarzSubmission}

\end{document}